\pdfoutput=1
\documentclass[11pt,a4paper]{article}
\usepackage[utf8]{inputenc}
\usepackage{amsfonts, amsmath, amssymb, color, a4wide}
\usepackage{hyperref}
\usepackage[utf8]{inputenc} % allow utf-8 input
\usepackage[T1]{fontenc}    % use 8-bit T1 fonts
\usepackage{url}            % simple URL typesetting
\usepackage{booktabs}       % professional-quality tables
\usepackage{amsfonts}       % blackboard math symbols
\usepackage{nicefrac}       % compact symbols for 1/2, etc.
\usepackage{microtype}      % microtypography

\usepackage{lmodern}

\usepackage{amsfonts}
\usepackage{amsmath}
\usepackage{amssymb}
\usepackage{bm}
\usepackage{algorithm}
\usepackage{algorithmic}
\usepackage{mathtools}
\usepackage{xspace}
\usepackage{xcolor}
\usepackage{amsthm}
\usepackage{cases}
\usepackage[numbers]{natbib}
\usepackage{dsfont}
\RequirePackage[shortlabels]{enumitem}

\newtheorem{theorem}{Theorem}[section]

\newtheorem{lemma}[theorem]{Lemma}
\newtheorem{corollary}[theorem]{Corollary}

\newtheorem{claim}[theorem]{Claim}

\newtheorem*{remark}{Remark}

\DeclarePairedDelimiter{\ceil}{\lceil}{\rceil}

\newtheorem{assumption}{Assumption}

\input{macros_gilles.tex}

\newcommand{\vertiii}[1]{{\left\vert\kern-0.25ex\left\vert\kern-0.25ex\left\vert #1 
		\right\vert\kern-0.25ex\right\vert\kern-0.25ex\right\vert}}

\title{Fast rates for prediction with limited expert advice}
\author{El Mehdi Saad$^{1}$ , Gilles Blanchard$^{1,2}$ \\
  $^1$Laboratoire de Mathématiques d'Orsay, CNRS, Université Paris-Saclay; $^2$Inria\\}
\date{}

\begin{document}
	\maketitle

	\begin{abstract}
		We investigate the problem of minimizing the excess generalization error with respect to the best expert prediction
		in a finite family in the stochastic setting,  under limited access to information. We assume that the learner  only has access to a limited number of expert advices per training round, %can only use a limited number of experts
		as well as for prediction.  Assuming that the loss function is Lipschitz and strongly convex, we show that if we are allowed to see the advice of only one expert per round for $T$ rounds in the training phase, or to use the advice of only one expert for prediction in the test phase, the worst-case excess risk is ${\Omega}(1/\sqrt{T})$ with probability lower bounded by a constant. However,
		if we are allowed to see at least two actively chosen expert advices per training round and use at least two experts for prediction, the fast rate $\mathcal{O}(1/T)$ can be achieved. We design novel algorithms achieving this rate in this setting, and in the setting where the learner has a budget constraint on the total number of observed expert advices,  and give precise instance-dependent bounds on the number of training rounds and queries needed to achieve a given generalization error precision.
		% We investigate the problem of minimizing the generalization error with respect to the best expert in a finite family in the stochastic setting,  under limited access to information. We consider that the learner has access to a fixed number of experts advice per round and can use only a fixed number of experts for prediction. Assuming that the loss function is Lipschitz and strongly convex, we show that if we are allowed to see the advice of only one expert per round for exploration, or to use the advice of only one expert for prediction, we cannot have an excess risk better than $\mathcal{O}(\frac{1}{\sqrt{T}})$ in \textit{deviation}. However, if we are allowed to see at least two actively chosen experts advices per round when training and use one at least two experts for prediction, we can have a rate of $\mathcal{O}(\frac{1}{T})$. We design novel algorithms achieving this rate, and analyse the impact of information restriction on the performance of generalization error. 
	\end{abstract}


\section{Introduction and setting}

We consider a generic prediction problem in a stochastic setting:
a target random variable $Y$ taking values in $\cY$ is to be predicted by a user-determined forecast $F$, also modeled
as a random variable, taking values in
a closed convex subset $\cX$ of $\mbr^d$. The mismatch between the two is measured via a loss function
$l(F,Y)$. The quality of the agent's output is measured by its generalization risk 
\begin{equation*}
  R(F) := \e[1]{ l\paren[0]{F, Y}}.
\end{equation*}
To assist us in this task, the forecast or ``advice'' of a number of ``experts''
$(F_1,\ldots,F_K)$ (also modeled as random variables) can be
requested. The agent's objective is to achieve a risk as close
as possible to the risk of the best expert
$R^*=\min_{i \in \intr{K}} R(F_i)$ (for a nonnegative integer $n$, we denote
$\intr{n}=\set{1,\ldots,n}$ ).
We measure the
performance of the user's forecast via its excess risk (or average regret) with
respect to that best expert.

The literature on expert advice generally considers the {\em cumulative} regret over a sequence of forecasts $F_t$ followed by
observation of the target variable $Y_t$ and incurring the loss $l(F_t,Y_t)$, $t=1,\ldots,T$. In the present work we will
separate observation (or training) phase and forecast phase: the user is allowed to observe (some of) the expert's predictions 
and the target variable for a number of independent, identically distributed rounds
$\left(Y_t, F_{1,t}, \dots, F_{K,t}\right)_{1\leq t \leq T}$
following certain rules to be specified. After the observation phase,
the user must decide of a prediction strategy, namely a convex combination of the experts
$\wh{F} = \sum_{i=1}^k \wh{w}_i F_i$, where the weights $\wh{w}_i$ can be chosen
based on the information gathered in the training phase. The risk of this strategy is $R(\wh{F})$, where the risk is evaluated
on new, independent data. In other words, if the training phase takes place over $T$ independent rounds, the
forecast risk is the expected loss over the $(T+1)$th, independent, round.

%	In many real-world applications, a forecaster has access to many expert advices in order to make a prediction.
In some situations, it may be overly expensive to query the advice of all experts at each round.
The cost can be monetary if each expert demands to be paid to reveal his opinion, possibly because they have
access to some information that others do not. In this case we may have a total limit
on how much we can spend. In a different context,
it is unrealistic to ask for the advice of all available doctors or to run a large battery of tests on each patient. In this case, we may be have a strong limit on the number of expert opinions
that can be consulted for each training instance.
In a more typical machine learning scenario, each ``expert'' might be a fixed prediction method $F_i=f_i(X)$
(using the information of a covariate $X$), where the predictor functions $f_i$
have been already trained in advance, albeit based on different sets of parameters or methodology; the goal then amounts to predictor selection or aggregation, in a situation where the computation of each single prediction constitutes the bottleneck cost, rather than data acquisition. %In such a situation,
Overall the agent's goal is to achieve a risk close to optimal while sparing on the number of experts queries --
both at training time and for forecast.
% Moreover, people are generally reluctant to undergo a battery of tests to have a diagnosis.
%One then wants to select a small number of opinions  and make decisions based on their outcome.
% This is an instance of {\em efficient model selection} problem.

Motivated by these questions we investigate several
scenarios for prediction with limited access to expert
advice. Furthermore, our emphasis is on obtaining {\em fast
  convergence rates} guarantees on the excess risk (i.e. $O(1/T)$ or $O(1/C)$, where $C$ is
the total query budget). These are possible
under a strong convexity assumption of the loss, specified below. Our contributions are the following.

\begin{itemize}
\item As a preliminary, we revisit (Section~\ref{sec:fullinfo}) the {\em full information setting}, with no limitations on queries. Maybe surprisingly, we contribute a new algorithm that is both simpler than existing ones and for which the proof of the fast convergence rate for excess risk is also
  elementary. Furthermore, for forecast we only need to consult 2 experts. The general principle of
  this algorithm will be reused in the limited observation settings.
\item We then investigate (Section~\ref{sec:bud}) the {\em budgeted setting} where we have a total
  query budget constraint $C$ for the training phase; then (Section~\ref{sec:lim}) the {\em two-query} setting where the agent is limited to $m=2$ queries per training round. In both cases,
  we give precise efficiency guarantees on the number of training expert queries needed to achieve a given precision for forecast. The obtained bounds come both in {\em instance-independent} (agnostic) and
  {\em instance-dependent} (depending on the experts' structure) flavors.
\item Finally, we give some lower bounds (Section~\ref{sec:low}) were we show that fast rates cannot
  be achieved if the agent is only allowed to consult one single expert per training round
  {\em or} for forecast.
\end{itemize}

The following assumption on the loss will be made throughout the paper:
\begin{assumption}\label{assump}
  $\forall y \in \mathcal{Y}$: $x \in \cX \subseteq \mbr^d \mapsto l\left(x,y\right)$ is $L$-Lipschitz and $\rho$-strongly convex.
\end{assumption}
Recall that a function $f: \mathcal{X} \to \mathbb{R}$ is $L$-Lipschitz if $\forall x,y \in \mathcal{X}$:$\left| f(x) - f(y)\right| \le L \norm{x-y}$, and $\rho$-strongly convex if the function: $x \to f(x)-\frac{\rho^2}{2} \norm{x}^2$ is convex.

	{\bf Remarks.} Assumption~\ref{assump} implies that the diameter of $\mathcal{X}$ is bounded
	by $8 L/\rho^2$ and the quantity $\sup_{x,x' \in \mathcal{X}, y\in \mathcal{Y}} \left| l(x,y) - l(x',y)\right|$ is bounded by $B:= 8 L^2/\rho^2$ (this notation shorthand will be used
        throughout the paper). Consequently, without loss of generality
	we can assume that the loss is bounded by $B$ (see Lemma~\ref{assump_cons} and subsequent
	discussion for details).
	% \item
	It is satisfied, for example, in the following setting: least square loss $l(x,y) = (y-x)^2$ where $x \in \mathcal{X}$ and $y \in \mathcal{Y}$ with $\mathcal{X}$ and $\mathcal{Y}$ are bounded subsets of $\mathbb{R}^d$.
	% \item
	Prior knowledge on $\rho$ is not necessary if $L$ and an upper bound on the the $l_{\infty}$ norm of the target variable $Y$ and the experts are known.

	\section{Discussion of related Work} %and Main Contributions}

	\textbf{Games with limited feedback (slow rates):} Our work investigates what happens between the full information and single-point feedback games. Learning with a restricted access to information was considered  under various settings in \cite{ben1998learning}, \cite{madani2004active}, \cite{guha2007approximation}, \cite{mannor2011bandits}, \cite{audibert2010regret}. A setting close to ours was considered in \cite{seldin2014prediction},
	where the agent chooses in each round a subset of experts to observe their advice, then follows the prediction of one expert. To minimize the cumulative regret in the adversarial setting, they used an extension of the Exp3 algorithm, which allows to have an excess risk of $\mathcal{O}(\sqrt{1/T})$ in the limited feedback setting and $\mathcal{O}(\sqrt{\log(C)/C})$ in the budgeted case with a budget $C$.
	
	The differences in the setting considered here is that (a) we are interested in the generalization
	error in the stochastic setting rather than the cumulative regret in an adversarial setting and
	(b) our assumptions of the convexity of the loss allow for the possibility of fast excess risk convergence.
	Moreover, we consider the more general case where the player is allowed to combine $p$ out of $K$ experts for prediction. The possibility of playing a subset of arms was considered in the literature of Multiple Play Multi-armed bandits. It was treated with a budget constraint in \cite{zhou2018budget} for example (see also \cite{xia2016budgeted}), where at each round, exactly $p$ out of $K$ possible arms have to be played. In addition to observing the individual rewards for each arm played, the player also learns a vector of costs which has to be covered with an a-priori defined budget $C$. In the stochastic setting, a UCB-type procedure gives a bound for the cumulative regret of $\mathcal{O}(\Delta_{\min}^{-1}\log(C)/C)$ that holds only in expectation, where $\Delta_{\min}^{-1}$ denotes the gap between the best choice of arms and the second best choice. This bound leads to an instance dependent bound of $\mathcal{O}(\sqrt{\log(C)/C})$ in the worst case. In the adversarial setting, an extension of Exp3 procedure gives a bound of $\mathcal{O}(\sqrt{\log(C)/C})$ for the cumulative regret that holds with high probability. In another online problem, where the objective is to minimize the cumulative regret in an adversarial setting with a  small effective range of losses, \cite{gerchinovitz2016refined} have shown the impossibility of regret scaling with the effective range of losses in the bandit setting, while \cite{thune2018adaptation} showed that it is possible to circumvent this impossibility result if the player is allowed one additional observation per round. However, in the settings considered, it is impossible to achieve a regret dependence on $T$ better than the rate of $\mathcal{O}(1/\sqrt{T})$.

	\textbf{Fast rates in the full information setting:} The learning task of doing as well as the best expert of a finite family in the sense of generalization error has been studied quite extensively in the full information case. %(\cite{cesa2006prediction}). %Progressive mixture rules and its variants are known to achieve a rate of $\mathcal{O}\left({1}/{T}\right)$ in {\em expectation},  where
	%$T$ is the number of training rounds (each round corresponding to a single prediction instance,
	%or data point). However, these procedures only achieve slow rates as far as {\em deviations} are concerned (\cite{audibert2007progressive}). New algorithmic ideas were required to develop procedures enjoying fast rates with high probability, such as the empirical star algorithm \cite{audibert2007progressivesupp} and aggregation via empirical risk minimization \cite{lecue2009aggregation}.
	In an adversarial setting, it is well-known that under suitable assumptions on the loss function (typically
	related to strong convexity), an appropriately tuned exponential weighted average (EWA) strategy has cumulative regret bounded by the ``fast rate''
	$\mtc{O}(\log(K)/T)$ \cite{haussler1998sequential,cesa2006prediction,audibert2009fast}, which, combined with the online-to-batch conversion principle \cite{cesabianchi2004generalization,audibert2009fast} (also known as
	progressive mixture rule, \cite{catoni97mixture,yang1999information}), yields a bound of the same order for the {\em expected} excess prediction risk in the stochastic case. However, it was shown that progressive mixture type rules are {\em deviation suboptimal}
	for prediction \cite{audibert2007progressive}, that is, their excess risk takes a value larger than $c/\sqrt{T}$ with
	constant positive probability over the training phase. To lift the apparent contradiction between the two last statements, consider that the
	excess risk of the EWA can take {\em negative} values, since it is an {\em improper} learning rule. Thus
	negative and positive ``large'' deviations can compensate each other so that the expectation is small.
	The inefficiency of EWA in deviation is a significant drawback, and alternatives to the EWA progressive mixture rule that achieve $\mtc{O}(\log(K)/T)$ excess prediction risk with high probability were proposed by \cite{lecue2009aggregation} and \cite{audibert2007progressivesupp}. In \cite{lecue2009aggregation}, the strategy consists in whittling down the set of experts
	by elimination of obviously suboptimal experts, and performing empirical
	risk minimization (ERM) over the convex combinations of the remaining experts.
	In \cite{audibert2007progressivesupp}, the {\em empirical star} algorithm consists
	in performing an ERM over all segments consisting of a two-point convex combination of
	the ERM expert and any other expert.
	Note that the empirical star algorithm has the advantage that the final prediction rule
	is a convex combination of (at most) {\em two} experts.

	\textbf{Linear regression with partially observed attributes:} Other related work is that of \cite{cesa2011efficient}, and \cite{hazan2011optimal} on learning linear regression models with partially observed attributes. The most related setting to ours is the local budget setting, where the learner is allowed to output a linear combination of features for prediction. The key idea is to use the observed attributes in order to build an unbiased estimate of the full information sample, then to use an optimization procedure to minimize the penalized empirical loss. In our setting, the minimization of penalized empirical loss was shown to be suboptimal (see \cite{lecue2007suboptimality}). Moreover, while we want to predict as well as the best expert, in \cite{cesa2011efficient}, the objective is to be as good as the best linear combination of features with a small additive term (the optimal rate, in this case, is $\mathcal{O}\paren[1]{{1}/{\sqrt{T}}}$). Finally, we consider that the restriction on observed attributes (experts advice) does not apply only to the training samples but also to the testing data.

	\textbf{Online convex optimization with limited feedback:} The idea of using multiple point feedback to achieve faster rates appeared in the online convex optimization literature (see \cite{agarwal2010optimal}, and \cite{shamir2017optimal}). It was shown that in the setting where the adversary chooses a loss function in each round if the player is allowed to query this function in two points, it is possible to achieve minimax rates that are close to those achievable in the full information setting. The key idea is to build a randomized estimate of the gradients, which are then fed into standard first-order algorithms. These ideas are not convertible into our setting because we consider a non-convex set of experts.

	\section{The full information case}\label{sec:fullinfo}
	
	In this section, we revisit the ``classical'' case where there is no constraint on the number of expert queries per
	observation round; assume the output of all experts are observed for $T$ rounds (in other words, $T$ i.i.d. training examples), which is the full information
	or ``batch'' setting.
	We want to output a final prediction rule with prediction risk controlled with high probability
	over the training phase.

	We start with putting forward an apparently new rule , simpler than existing ones \cite{lecue2009aggregation,audibert2007progressivesupp}, for the full information setting which, like the empirical star \cite{audibert2007progressivesupp},
	outputs a convex combination of two experts. In contrast to the latter, our rule does not need any optimization over a union of segments. The underlying principle will guide us to
	construct a budget efficient expert selection rule in the sequel.
	
	Define $\hat{R}(F_i) := T^{-1} \sum_{t=1}^T l(F_{i,t},Y_t)$ the empirical loss of expert $i$, and
	$\hat{d}_{ij} := (T^{-1} \sum_{t=1}^T (F_{i,t}-F_{j,t})^2)^{\frac{1}{2}}$ the empirical $L_2$ distance between experts $i$ and $j$
	over $T$ rounds. Finally let $\alpha=\alpha(\delta) := (\log(4K\delta^{-1})/T)^{\frac{1}{2}}$, where $\delta\in (0,1)$ is a fixed
	confidence parameter. Define
	%        Let $\eta>0$ be a constant to be fixed later, and define
	\begin{equation}
	\Delta_{ij} := \hat{R}(F_j) - \hat{R}(F_i) - 6\alpha  \max\set[1]{ L\hat{d}_{ij}, B\alpha }.
	\end{equation}
	The quantity $\Delta_{ij}$ can be interpreted as a test statistic: if $\Delta_{ij}>0$, then we have a guarantee that
	$R(F_j)>R(F_i)$, so that expert $j$ is sub-optimal; this guarantee holds for all $(i,j)$ uniformly with probability $(1-\delta)$.
	It therefore makes sense to reduce the set of candidates to
	\begin{equation} \label{eq:defs}
	S:=\set[2]{ j \in \intr{K}: \sup_{j \in \intr{K}} \Delta_{ij} \leq 0}.
	\end{equation}
	Our new full information setting rule is the following: 
	\begin{equation}
	\label{eq:fullinfrule}
	\text{ choose } \bar{k}\in S  \text{ arbitrarily ; \;\; pick } \bar{j} \in \argmax_{j \in S} \hat{d}_{\bar{k}j}; \;\; 
	\text{ predict } \wh{F} := \frac{1}{2}(F_{\bar{k}} + F_{\bar{j}}).
	\end{equation}
	In words, the above rule consists in eliminating all experts that are manifestly outperformed by another
	one, and, among the remaining experts, pick two that disagree as much as possible (in terms of empirical $L^2$
	distance~) and output their simple average for prediction.
	The next theorem establishes fast convergence rate for the excess risk of this rule:
	\begin{theorem}\label{thm:fullinf}
		If Assumption~\ref{assump} holds and $\delta\in(0,1)$ is fixed, then for the prediction rule $\wh{F}$ defined by~\eqref{eq:fullinfrule}, it holds
		with probability $1-3\delta$ over the training phase ($c$ is an absolute constant):
		\begin{equation*}
		R(\wh{F}) \le R^*+cB \frac{\log(4K\delta^{-1})}{T}.
		\end{equation*}
		%           with $c$ an absolute constant.
		% where $\kappa$ is a constant depending only on $L$ and $\rho$.
	\end{theorem}
	\begin{proof}
		Let $d^2_{ij} = \e{(F_i-F_j)^2}$. The result hinges on the following
		high confidence control of risk differences, established
		in Corollary~\ref{cor:deltacontrol} as a direct consequence of the
		empirical Bernstein's inequality:
		%Since $T$ is fixed, we drop the dependence of $T$ in all notation such as
		%$\wh{R}(i) = \wh{R}(i,T)$  and also abbreviate $\alpha=\alpha(T,\delta)$.
		with probability at least $1 - 3\delta$, it holds:
		\begin{align}
		\label{eq:deltacontrol}
		\text{For all } i, j \in \intr{K}: \qquad 
		\Delta_{ij} & \leq (R_j - R_i) \leq \Delta_{ij} + 32\alpha \max\paren{ L d_{ij}, B \alpha}.          % &\le \sqrt{2}L~\hat{d}_{ij}~\alpha_{ij}(t, \delta) + 3B~\alpha^2_{ij}(t, \delta) \\
		%          \abs[1]{\hat{d}_{ij}^2 - d_{ij}^2} &\le 2B \alpha \max\paren{ d_{ij}, 3B \alpha}.
		\end{align}
		Let $i^* \in \argmin_{i \in \intr{K}} R_i$ be an optimal expert. Since $R_{i^*}-R_j\leq 0$
		for all $j\in \intr{K}$,
		it follows that if~\eqref{eq:deltacontrol} holds, then $i^* \in S$, from
		the definition of $S$.
		% ($S$ being defined by~\eqref{eq:defs}).
		So if~\eqref{eq:deltacontrol} holds, we have
		\begin{align*}
		R\paren{\frac{F_{\bar{k}}+F_{\bar{j}}}{2}}
		& \leq \frac{1}{2}\paren[1]{R_{\bar{k}} + R_{\bar{j}}} - \frac{\rho^2}{8} d_{\bar{k}\bar{j}}^2\\
		& = R^* + \frac{1}{2}\paren{(R_{\bar{k}} - R_{i^*}) + (R_{\bar{j}} - R_{i^*}) }- \frac{\rho^2}{8} d_{\bar{k}\bar{j}}^2\\
		& \leq R^* + \frac{1}{2}\paren{\Delta_{\bar{k}i^*} + \Delta_{\bar{j}i^*} } +
		16\alpha \paren{\max\paren{ L d_{\bar{j}i^*}, {B} \alpha} +
			\max\paren{  L d_{\bar{k}i^*}, B \alpha}}
		- \frac{\rho^2}{8} d_{\bar{k}\bar{j}}^2\\
		& \leq R^* + 32 B \alpha^2 
		+ 48 L \alpha d_{\bar{k}\bar{j}}
		- \frac{\rho^2}{8} d_{\bar{k}\bar{j}}^2;
		\end{align*}
		where we have used strong convexity of the loss (and therefore of $R(.)$ with respect to the
		$L^2$ distance) in the first line;
		the right-hand side of~\eqref{eq:deltacontrol} in the third line; and, in the last line,
		the fact that $\bar{j},\bar{k},i^*$ are all in $S$ along with
		$d_{\bar{j}i^*} \leq d_{\bar{j}\bar{k}} + d_{\bar{k}i^*}\leq 2 d_{\bar{j}\bar{k}}$ by
		construction of $\bar{j}$. Finally upper bounding the value of the last bound by its
		maximum possible value as a function of $d_{\bar{k}\bar{j}}$ and recalling $B=8L^2/\rho^2$, we obtain the statement.
	\end{proof} 
	
	% \begin{theorem}\label{thm:fullinf}
	% If Assumption~\ref{assump} holds, the prediction rule $\wh{F}$ defined by~\eqref{eq:fullinfrule}
	% 	with $\eta > 3L$ satisfies, with probability $1-2\delta$ over the training phase:
	% 	\begin{equation*}
	% 		R(\wh{F}) \le R^*+(2\eta+2.5\eta^2)B \frac{\log(4K\delta^{-1})}{T}.                
	% 	\end{equation*}
	% 	%		where $\kappa$ is a constant depending only on $L$ and $\rho$.
	% \end{theorem}

	\section{Budgeted Setting}\label{sec:bud}
	
	In this section, we consider the budgeted setting. More precisely, given an a-priori defined budget $C$, at each round the decision-maker selects an arbitrary subset of experts and asks for their predictions.
The choice of these experts may of course depend on past observations available to the agent.
        The player then pays a unit for each observed expert's advice. The game finishes when the budget is exhausted, at which point the player outputs a convex combination of experts for prediction.
	
	We convert the batch rule defined in the full information setting to an "online" rule by performing the test $\Delta_{ji} >0$ for each pair $(i,j)$ after each allocation. If at any round an expert $i \in \intr{K}$ fails any of these tests (i.e $\exists j: \Delta_{ji} >0$), it is no longer queried. This extension allows us to derive instance dependent bounds, which cover the rates obtained in the batch setting in the worst case.
	
	Since the tests $\Delta_{ij}>0$ are performed after each allocation, we introduce the following modification on the definition of $\Delta_{ij}$, for concentration inequalities to hold uniformly over the runtime of the procedure. We define $\Delta_{ij}(t, \delta)$ as follows:
	\begin{equation*}
	\Delta_{ij}(t, \delta) := \hat{R}(j, t) - \hat{R}(i, t) - 6 \alpha(t, \delta/(t(t+1))  \max\set[1]{L \hat{d}_{ij}(t), B\alpha(t, \delta/(t(t+1)) }.
	\end{equation*}     
	\begin{algorithm} 
		%\centering
		\caption{Budgeted aggregation \label{algo:budgeted}}
		\begin{algorithmic}
			\STATE \textbf{Input}  $ \delta$, $L$ and $\rho$.
			\STATE Initialization: $S \gets \intr{K}$.
			\FOR{ $T=1,2,\dots $ } 
			\STATE Jointly query all the experts in $S$ and update $\Delta_{ij}>0$ for all $i,j$.
			\STATE For all $i,j \in \intr{K}$, if $\Delta_{ij}>0$, eliminate $j$: $S \gets S \setminus \{j\}$.
			\IF{the budget is consumed }
			\STATE let $\bar{k} \in S$, and $\bar{l} \gets \underset{j \in S}{\text{argmax}}~\hat{d}_{\bar{k}j}$.
			\STATE Return $\frac{1}{2} \left(F_{\bar{k}} + F_{\bar{l}}\right)$.
			\ENDIF
			\ENDFOR 
			
		\end{algorithmic}
	\end{algorithm}
	
	Let $\mathcal{S}^{*} := \argmin_{i \in \intr{K}} R(F_i)$ denote the set of optimal experts. For $i,j \in \intr{K}$, we denote by $d_{ij}:= ( \mathbb{E}[\left(F_i - F_j\right)^2])^{1/2}$ the $L_2$ distance between the experts $F_i$ and $F_j$. For $i\in \intr{K}$, we introduce the following quantity:
	\begin{equation*}
	\Lambda_i := \min_{i^* \in \mathcal{S}^*}\max\left\lbrace \frac{L^2d^2_{ii^*}}{\left|R(F_i) - R(F_{i^*})\right|^2}; \frac{B}{R(F_i) - R(F_{i^*})} \right \rbrace.
	\end{equation*}
	
	Define the following set of experts:
	\begin{equation*}
	\mathcal{S}_{\epsilon} = \left\lbrace i \in \intr{K}: \Lambda_i > \frac{1}{\epsilon}\right \rbrace,
	\end{equation*}
	and let $\mathcal{S}_{\epsilon}^c$ be its complementary.

	\begin{theorem}\label{th:main3}(Instance dependent bound)
		Suppose Assumption~\ref{assump} holds. Let $C \ge K$ denote the global budget on queries and denote $\hat{g}$ the output of Algorithm~\ref{algo:budgeted} with inputs $(\delta, L, \rho)$ when the budget $C$ runs out.
		For any $\epsilon\ge  0$, if:
		\begin{equation*}
			C > 578 C_{\epsilon} \log\left(K\delta^{-1} C_{\epsilon}\right),
		\end{equation*}
		where
	\begin{equation*}
		C_{\epsilon} := \sum_{i \in \mathcal{S}_{\epsilon}^c} \Lambda_i + \left|\mathcal{S}_{\epsilon}\right|~\min \left\lbrace \frac{1}{\epsilon};  \Lambda^* \right \rbrace,
	\end{equation*}

	where $\Lambda^* := \max_{i:\Lambda_i<+\infty} \Lambda_i$, then, with probability at least $1-\delta$:
		\begin{equation*}
			R(\hat{g}) \le R^* +cB\epsilon,
		\end{equation*}
	where $c$ is an absolute constant. 									
	\end{theorem} 
	\begin{remark}
		Observe that the above result gives in particular a query
		budget bound for the problem of best expert identification
		in our setting, by taking $\epsilon=0$, in which case the required expert query budget
		is of order $\sum_{i: \Lambda_i < +\infty} \Lambda_i$ up to logarithmic terms.
		We can compare this to the problem of best arm identification in a bandit setting (one arm pull/query per round);
		our setting can be cast into that framework by considering each expert as an arm and only
		recording the information of the loss of the asked expert. The known optimal query bound for best arm identification in the classical multi-armed bandits setting with loss/reward bounded by $B$ is of order $\sum_{i: \Lambda_i < +\infty} \wt{\Lambda}_i$~\cite{kaufmann2016complexity}, where $\wt{\Lambda}_i = B^2(R(F_i) - R(F_{i^*}))^{-2}$.
		Since the diameter of $\mathcal{X}$ is bounded by $B/L$ (see Lemma~\ref{assump_cons}), it holds $\Lambda_i \leq \wt{\Lambda}_i$. Hence, for best expert identification, the bound of Theorem~\ref{th:main3}  improves upon the best arm identification bound, potentially by a significant margin
		(in particular concerning the contribution of  suboptimal but close to optimal experts for which $d_{ii^*} \ll B/L$ and $R_i - R_{i^*} \ll B$). Again, the improvement is due to the Assumption~\ref{assump} on the loss and the possibility
		to query several experts per round, which are not used when casting the problem as a classical bandit setting.
		%The condition $\eta >48 \max\left\lbrace L, \frac{L}{\rho^2}, \frac{L^2}{\rho^2}\right \rbrace$ can be substituted by the condition $\eta > \max\{6L, 4a\}$ where $a$ is a bound on the set $\mathcal{X}$ and the $\norm{.}_{\infty}$ norm of the loss function $l$.  
	\end{remark}
	
	\section{Two queries per round ($m =p = 2$)}\label{sec:lim}
	
	In this section, we suppose that the decision-maker is constrained to see only two experts' advice per round ($m=2$). We suppose that the horizon is unknown; when the game is halted, the player outputs a convex combination of at most two experts ($p=2$). We will show that the rates obtained are as good as in the full information case in its dependence on the number of rounds $T$.
	%We consider the case where the learner is allowed to query two experts advices per round and to play a combination of two experts for prediction. We show that this relaxation allows us to improve the $\mathcal{O}\left(\frac{1}{\sqrt{T}}\right)$ in the previous section and attain the rate of $\mathcal{O}\left(\frac{1}{T}\right)$.
	
	Algorithm~\ref{algo:bud} works as follows. 
	%To circumvent the limitation of observing only two experts per round, we select experts in a uniform way to see their predictions. 
	To circumvent the limitation of observing only two experts per round, in each round, we sample a pair $(i,j) \in S \times S$ in a uniform way, where $S$ is the set of non-eliminated experts. Then the tests $\Delta'_{ji} \le 0$ and $\Delta'_{ij} \le 0$ are performed, where $\Delta'_{ij}$ is defined by \eqref{eq:def_delta}. If $i$ or $j$ fail the test, which means that it is a suboptimal expert, it is eliminated from $S$.
	
	Finally, when the algorithm is halted,  depending on the number of allocated samples, we choose either an empirical risk minimizer over the non-eliminated experts or the mean of two experts from $S$ that are distant enough. This rule allows our algorithm's output to enjoy the best of converge rates of the two methods.
	
	We introduce the following notations: In round $t$, denote $T_{ij}(t)$ the number of samples where predictions of experts $i$ and $j$ were jointly queried and $T_i(t)$ the number of rounds where the prediction of expert $i$ was queried. Denote $\hat{R}_{ij}(j, t)$ the empirical loss of expert $i$ calculated using only the $T_{ij}(t)$ samples queried for $(i,j)$ jointly. We define $\alpha_{ij}(t, \delta) := \sqrt{\frac{\log(4K\delta^{-1})}{T_{ij}(t)}}$ if $T_{ij}(t)>0$ and $\alpha_{ij}(t) = \infty$ otherwise. Let $\hat{d}_{ij}(t)$ be the empirical $L_2$ distance between experts $i$ and $j$ based on the $T_{ij}(t)$ queried samples.
	Denote $\delta_t := \delta/(t(t+1))$. For $i,j \in \intr{K}$ we define:
	\begin{equation}\label{eq:def_delta}
	\Delta'_{ij}(t, \delta) := \hat{R}_{ij}(j, t) - \hat{R}_{ij}(i, t) - 6 \max\left \lbrace L \alpha_{ij}(t, \delta_t) \hat{d}_{ij}(t), B\alpha^2_{ij}(t, \delta_t) \right \rbrace.
	\end{equation} 	
	
	\begin{algorithm} 
		%\centering
		\caption{Two-point feedback \label{algo:bud}}
		\begin{algorithmic}
			\STATE \textbf{Input}  $\delta, L$ and $\rho$.
			\STATE Initialization: $S \gets \intr{K}$. 
			\FOR{ $T=1,2,\dots $ } 
			\STATE Let $(i,j) \in \argmin_{(u,v) \in S \times S} T_{uv}$.
			\STATE Query the advice of experts $i$ and $j$ and update the corresponding quantities.
			\STATE For all $u,v$: If $\Delta'_{uv}>0$: $S \gets S \setminus \{v\}$.
			\ENDFOR  
			\STATE \textbf{On interrupt:} Let $\hat{k} \in S$ and let $\hat{l} \gets \underset{j \in S}{\text{argmax}}~\hat{d}_{\hat{k}j}$.
			\STATE Let $\hat{q}$ denote the empirical risk minimizer on $S $.
			\IF{$T_{\hat{k}\hat{l}} > \sqrt{\log(KT\delta^{-1}) T_{\hat{q}}}$}
			\STATE Return $\frac{1}{2} \left(F_{\hat{k}} + F_{\hat{l}}\right)$.
			\ELSE 
			\STATE Return $F_{\hat{q}}$.
			\ENDIF
		\end{algorithmic}
	\end{algorithm}

	Our first result in this setting is an empirical bound. 
	At any interruption time, it gives a bound on the excess risk, only depending on
	quantities available to the user, using the number of queries resulting from the querying strategy in Algorithm~\ref{algo:bud}. We then use a worst-case bound on these quantities to develop an instance independent bound in Corollary~\ref{th:main1}. 
	\begin{theorem}\label{th:main0}(Empirical bound)
		Suppose Assumption~\ref{assump} holds. Let $T \ge 2K^2$, and denote $\hat{g}$ the output of Algorithm~\ref{algo:bud} with inputs $(\delta, L, \rho)$ in round $T$. Then with probability at least $1-3\delta$:
		\begin{equation} \label{eq:emp}
		R\left(\hat{g}\right) \le R^* + c~B  \min \left \lbrace\frac{\log\left(TK\delta^{-1}\right)}{T_{\hat{k}\hat{l}}(T)}, \sqrt{\frac{\log\left(TK\delta^{-1}\right)}{T_{\hat{q}}(T)}} \right \rbrace,
		\end{equation}
		where $\hat{k}, \hat{l}$ and $\hat{q}$ are the experts in Algorithm~\ref{algo:bud} and $c$ is an absolute constant. 
	\end{theorem}
	\paragraph{Proof Sketch of Theorem~\ref{th:main0}}
	
	We start by noting that when running Algorithm~\ref{algo:bud}, the optimal experts $\mathcal{S}^{*} = \argmin_{i \in \intr{K}} R(F_i)$ are never eliminated with high probability (Lemma~\ref{lem:i*}). This shows in particular, that when the procedure is terminated, we have $\mathcal{S}^{*} \subseteq S_T$,
        where $S_T$ is the set of non-eliminated experts at round $T$.
	
	Then we show the following key result: in each round $t \le T$, for any expert $i\in S_t$, let $j\in \argmax_{l \in S_t} \hat{d}_{il}(t)$, we have with probability at least $1-\delta$:
	\begin{equation*}
	R\left(\frac{F_i + F_j}{2}\right) \le R^* + cB~\frac{\log(K\delta_t^{-1})}{T_{ij}(t)}.
	\end{equation*}
	For the second bound, recall that $i^*$ belongs to $S_T$ with high probability. Therefore, performing an empirical risk minimization over the set of non-eliminated experts leads to the bound $\sqrt{\frac{\log(KT\delta^{-1})}{T_q(T)}}$, through a simple concentration argument using Hoeffding's inequality.

	\begin{corollary}\label{th:main1}(Instance independent bound)
		Suppose assumption 1 holds. Let $T \ge 2K^2$, and denote $\hat{g}$ the output of Algorithm~\ref{algo:bud} with inputs $(\delta, L, \rho)$ in round $T$. Then with probability at least $1-3\delta$:
		\begin{equation*}
		R\left(\hat{g}\right) \le R^* + c~B  \min \left \lbrace\frac{K^2\log\left(TK\delta^{-1}\right)}{T}, \sqrt{\frac{K\log\left(TK\delta^{-1}\right)}{T}} \right \rbrace,
		\end{equation*}
		where $c$ is an absolute constant. 
	\end{corollary}
	\begin{proof}
		We develop an elementary bound on $T_{\hat{k}\hat{l}}$ and $T_{\hat{q}}$, then we inject these bounds into inequality \eqref{eq:emp}.
		
		Note that: $\hat{q}, i^* \in S_T$, hence $T_{\hat{q}}(T), T_{i^*}(T) \ge \frac{T}{2K}$. Moreover, we have:
		\begin{equation*}
			T_{\hat{k}\hat{l}}(T) \ge \frac{T}{K^2}.
		\end{equation*}
		Using inequality \eqref{eq:emp}, we obtain the result.
	\end{proof}
	
	\begin{remark}
		%If $T>K^3$, we need $T = \Omega(K^2/\epsilon)$ to attain an excess risk of $\epsilon$, while in the full information setting $T = \Omega(1/\epsilon)$ samples are sufficient. We think that this gap is inevitable as the number of jointly observed pairs $(F_{i,t}, F_{j,t})$ are $T$ in this limited setting, while we observe $K^2 T$ such pairs in the batch case.
          Observe that in all the considered settings (full information, budgeted and limited advice), the number of jointly sampled pairs $(F_i, F_j)$ to attain an excess risk of $\mathcal{O}(\epsilon)$ is of the order of $\mathcal{O}(K^{2}/\epsilon)$. Being able to ask a set of $m$ experts simultaneously in a training round allows to sample $m(m-1)/2$ pairs for a query cost of $m$: this is the advantage of the budgeted setting, while we have to query each pair in succession under the strict $m=2$ constraint, resulting in
          a higher cost overall.          
	\end{remark}
	
	%	\begin{theorem}\label{th:main2}(Instance dependent bound)
	%		Suppose assumption 1 holds. Let $\hat{g}$ denote the output of Algorithm~\ref{algo:bud} with input $(\delta, \eta, L, \rho)$ and $T$ denote the total number of rounds. If $\eta > 3L$, and:
	%		\begin{equation*}
	%			T > \kappa~ K\log\left(\frac{K}{\delta} \min\left\{ \frac{1}{\epsilon}, \Lambda^{*} \right\}\right) \left( \sum_{i\in \mathcal{S}^c_{\epsilon}} \Lambda_{i} + \min\left\{\frac{1}{\epsilon}, \Lambda^{*} \right\} \right),
	%		\end{equation*}
	%		where $\Lambda^{*} = \max\left\{ \Lambda_{i}; \Lambda_{i} <+\infty\right\}$ and $\kappa$ is a constant depending only on $\eta$, $\rho$ and $L$. We have, with probability at least $1-\delta$:
	%		\begin{equation*}
	%			R(\hat{g})  \le R^* + \epsilon.
	%		\end{equation*}
	%	\end{theorem}
	
	%	\textcolor{red}{Proposition de correction:}
	\begin{theorem}\label{th:main2}(Instance dependent bound)
		Suppose Assumption~\ref{assump} holds. Let $\hat{g}$ denote the output of Algorithm~\ref{algo:bud} with input $(\delta, L, \rho)$ and $T$ denote the total number of rounds. Let $\epsilon >0$, if :
		\begin{equation*}
		T \ge 578~C_{\epsilon}   \log\left(\delta^{-1} C_{\epsilon}\right),
		\end{equation*}
		where
		\begin{equation*}
			C_{\epsilon} := K\sum_{i \in \mathcal{S}_{\epsilon}^c}\Lambda_{i} + 2 \left|\mathcal{S}_{\epsilon}\right|^2 \min\left\lbrace\frac{1}{\epsilon}, \Lambda^* \right \rbrace,
		\end{equation*}
		where $\Lambda^* := \max_{i:\Lambda_i<+\infty} \Lambda_i$, then, with probability at least $1-\delta$:
		\begin{equation*}
		R(\hat{g})  \le R^* + cB~ \epsilon,
		\end{equation*}
		where $c$ is an absolute constant.
	\end{theorem}

	\begin{remark}
		%This instance dependent bound, which recovers the bound in Corollary~\ref{th:main1} in the worst case, is a consequence of the allocation strategy in Algorithm~\ref{algo:bud}. Had we used a uniform allocation strategy over all pairs $(F_i, F_j)$ and apply the full-information bound, we would lose the ability of our algorithm to adapt to the case where many experts are clearly suboptimal.\\ 
		If the algorithm is allowed to query $m>2$ expert advices per round, then it can be modified to attain an improved excess risk. We present this extension in Section~\ref{sec:m3} in the appendix, and prove that it leads to a rate of $\mathcal{O}\left(\frac{(K/m)^2}{T} \log(KT/\delta)\right)$ , which interpolates for intermediate values of $m$.
	\end{remark}
	
	\paragraph{Proof Sketch of Theorem~\ref{th:main2}}
	
	First, we develop instance-dependent upper and lower bound for $T_{ij}(t)$, for any $i,j \in \intr{K}$ such that: $R(F_i) \neq R(F_j)$. To do this we introduce the following lemma (see Lemma~\ref{lem:1} in the appendix):
	\begin{lemma}
          Let $i,j \in \intr{K}$ such that $R(F_i) \neq R(F_j)$. With probability at least $1-4\delta$, for all $t \ge 1$, if
		\begin{equation*}
		T_{ij}(t) \ge 289 \log\left(K \delta_t^{-1}\right) \max\left\lbrace \frac{L^2d_{ij}^2}{\left|R(F_i) - R(F_j)\right|^2}; \frac{B}{\left|R(F_i) - R(F_j)\right|}\right \rbrace,
		\end{equation*}
		then we have either $\Delta'_{ij} > 0$ or $\Delta'_{ji} > 0$;
				furthermore, if
		\begin{equation*}
		T_{ij}(t) \le  3 \log\left(K \delta_t^{-1}\right) \max\left\lbrace \frac{L^2d_{ij}^2}{\left|R(F_i) - R(F_j)\right|^2}; \frac{B}{\left|R(F_i) - R(F_j)\right|} \right\rbrace,
		\end{equation*}
		then we have: $\Delta'_{ij} \le 0$ and $\Delta'_{ji} \le 0$.
	\end{lemma}

	This lemma gives in particular an upper bound on the number of allocations needed for an expert $i$ to be eliminated by an optimal expert $i^*$ (i.e. to fail the test $\Delta_{ii^*} \le 0$). Then, we derive a bound on the number of rounds $T_{\epsilon}$ required to eliminate all the experts in $\mathcal{S}_{\epsilon}^c$ and we conclude by showing that $T-T_{\epsilon}$ is large enough to ensure that the experts $\hat{k}$ and $\hat{l}$ in algorithm~\ref{algo:bud} satisfy $T_{\hat{k}\hat{l}} > 1/\epsilon$ with high probability.

	%		\section{Intermediate case: $m\ge 3, p=2$}\label{sec:m3}
	%	
	%	In this section we assume that the learner is allowed to access more than two experts advices per round. We show that this leads to an improvement of the bound in theorem~\ref{th:main1}. We consider the following extension of algorithm~\ref{algo:bud}: 
	%	
	%	
	%	
	%	\begin{algorithm}[H] 
	%		%\centering
	%		\caption{Intermediate case \label{algo:bud3}}
	%		\begin{algorithmic}
	%			\STATE \textbf{Input} $m$, $L$ and $\rho$.
	%			\STATE Initialization: $S \gets \intr{K}$. 
	%			\FOR{ $T=1,2,\dots $ } 
	%			\STATE Sample a subset $\mathcal{M}$ of size $m$ from $\intr{K}$ uniformly at random.
	%			\STATE Query the advice of experts in $\mathcal{M}$ and update the corresponding quantities.
	%			\STATE For all $i,j$: If $\Delta'_{ij}>0$: $S \gets S \setminus \{j\}$.
	%			\ENDFOR  
	%			\STATE \textbf{On interrupt:} Let $\hat{k} \in S$ and let $\hat{l} \gets \underset{j \in S}{\text{argmax}}~\hat{d}_{\hat{k}j}$.
	%			\STATE Return $\frac{1}{2} \left(F_{\hat{k}} + F_{\hat{l}}\right)$.
	%		\end{algorithmic}
	%	\end{algorithm}
	%	
	%	
	%	\begin{theorem}\label{th:maingen}(Instance independent bound)
	%		Suppose assumption 1 holds. Let $T \ge 1$, and denote $\hat{g}$ the output of algorithm~\ref{algo:bud3} with inputs $(m, L, \rho)$ in round $T$. If $m\ge 3$, then with probability at least $1-\delta$:
	%		\begin{equation*}
	%		R\left(\hat{g}\right) \le \min_{i \in \intr{K}} R_i + cB~  \frac{(K/m)^2\log\left(2TK\delta^{-1}\right)}{T},
	%		\end{equation*}
	%		where $c$ is an absolute constant. 
	%	\end{theorem}

	\section{Lower Bounds for $m=1$ or $p=1$}\label{sec:low}
	This section considers the case where the agent is restricted to selecting one expert at the end of the procedure ($p=1$), and the case where the learner is restricted to see only one feedback per round ($m=1$). We show that in either case it is impossible to do better than an excess risk $\mathcal{O}\paren[1]{{1}/{\sqrt{T}}}$ in deviation.
	
	Lemma~\ref{lem:lower1} is a direct consequence of a more general lower bound in \cite{lee1998importance}, which proved that if the closure of the experts class is non-convex, and a single expert must be picked at the end (``proper'' learning rule), then even under full information access during training the best achievable rate with high probability is $\mathcal{O}\paren[1]{{1}/{\sqrt{T}}}$.

	\begin{lemma}($p=1$)\label{lem:lower1}
		Consider the squared loss function. For $K=m=2$ and $p=1$, for any $T>0$,
		and for any convex combination of the experts $\hat{g}$ output after $T$ training rounds,
		there exists a probability distribution for experts $\left\lbrace F_1, F_2 \right \rbrace$ and target variable $Y$ (all bounded by $1$) such that,
		with probability at least $0.1$,
		\begin{equation*}
		\hat{R}_T\left(\hat{g}\right) - R^* %\min\{R_1;R_2\}
		\ge \frac{c_1}{\sqrt{T}},
		\end{equation*}
		where $c_1>0$ is an absolute constant. 
	\end{lemma}
	
	The second result shows that the same lower bound holds for the bandit feedback ($m=1$) setting, even if the learner is allowed to predict using a convex combination of all the experts at the end. To the best of our knowledge, this is the first lower bound for deviations in this setting. 
	
	\begin{lemma}($m=1$)\label{lem:lower2}
		Consider the squared loss function. %There exists an absolute constant $c_2>0$ such that the following holds.
		For $K=p=2$, and $m=1$, for any $T>0$, for any convex combination of the experts $\hat{g}$ output after $T$ training rounds,
		there exists a probability distribution for experts $\left\lbrace F_1, F_2 \right \rbrace$ and target variable $Y$ (all bounded by $1$) such that  with probability at least $0.1$,
		\begin{equation*}
		\hat{R}_T\left( \hat{g}\right) - R^* %\min\{R_1;R_2\}
		\ge \frac{1}{2\sqrt{T}}.
		\end{equation*} 
	\end{lemma}
	
	\section{Conclusion}
	
	We discussed the impact of restricted access to information in generalization error minimization with respect to the best expert. As many classical methods, such as progressive mixture rules
        (and randomized versions thereof) are deviation suboptimal, we proposed a new procedure achieving fast rates with high probability. We focused on the global budget setting, where a constraint on the total number of expert  queries is made, and the local budget, where a limited number of expert advices are shown per round. Moreover, we proved fast rates are impossible to achieve if the agent is allowed to see just one expert advice per round or choose just one expert for prediction.
	
	An interesting future direction is allowing experts to learn from data during the process. In this case, the i.i.d. assumption on the loss sequence is dropped, which necessitates deriving a new concentration for the key quantities. 
	
        {\bf Acknowledgements}

                We acknowledge support from the Agence Nationale de la Recherche (ANR),
        ANR-19-CHIA-0021-01 ``BiSCottE''; and the
        Franco-German University (UFA) through the binational Doktorandenkolleg CDFA 01-18.

\bibliographystyle{abbrv}
\bibliography{bib_data_base}

\newpage

\appendix

{\bf \Large Appendix: proofs}

	\section{Notation}

The following notation pertains to all the considered algorithms,
where $t$ is a given training round:
\begin{itemize}
	\item Let $\mathcal{T}_i(t)$ denote the set of training round indices where the advice of expert $i$ was queried and let  $T_i(t) := \left| \mathcal{T}_i(t)\right|$.
	\item Let $\mathcal{T}_{ij}(t)$ denote the set of training round indices where the advice of experts $i$ and $j$ where jointly queried and let $T_{ij}(t):= \left| \mathcal{T}_{ij}(t) \right|$.
	\item Let $\hat{R}_{ij}(j,t)$ denote the empirical loss of expert $j$ calculated using only the $T_{ij}(t)$ samples queried for $(i,j)$ jointly:
	\begin{equation*}
	\hat{R}_{ij}(j, t) := \frac{1}{T_{ij}(t)} \sum_{s \in \mathcal{T}_{ij}(t)} l\left(F_{j,s}, Y_s\right).
	\end{equation*}
	\item $\hat{R}_i(t)$ denote the empirical loss of expert $i$ calculated using the $T_i(t)$ queried samples:
	\begin{equation*}
	\hat{R}_{i}(t) := \frac{1}{T_{i}(t)} \sum_{s \in \mathcal{T}_{i}(t)} l\left(F_{i,s}, Y_s\right).
	\end{equation*}
	\item Define $\alpha_{ij}(t, \delta) := \sqrt{\frac{\log(4K\delta^{-1})}{T_{ij}(t)}}$ if $T_{ij}(t)>0$ and $\alpha_{ij}(t) = \infty$ otherwise.
	\item Define $\alpha_i(t, \delta) := \sqrt{\frac{\log\left(4K\delta^{-1}\right)}{T_i(t)}}$ if $T_{i}(t)>0$ and $\alpha_{i}(t) = \infty$ otherwise.
	\item Let $\hat{d}_{ij}(t)$ denote the empirical $L_2$ distance between experts $i$ and $j$ based on the $T_{ij}(t)$ queried samples:
	\begin{equation*}
	\hat{d}^2_{ij}(t) := \frac{1}{T_{ij}(t)} \sum_{s \in \mathcal{T}_{ij}(t)} \left(F_{i,s} - F_{j,s}\right)^2.
	\end{equation*}
	\item Define $\Delta'_{ij}(t, \delta) := \hat{R}_{ij}(j, t) - \hat{R}_{ij}(i, t) - 6 \alpha_{ij}(t, \delta) \max\left \lbrace L  \hat{d}_{ij}(t), B \alpha_{ij}(t, \delta) \right \rbrace$.
	\item Let $d_{ij}$ denote the $L_2$ distance between experts $i$ and $j$:
	\begin{equation*}
	d_{ij} := \mathbb{E}\left[ \left(F_i - F_j\right)^2\right].
	\end{equation*}
	\item We denote $R(.)$ the expected risk function: $R(.) = \mathbb{E}[l(.,Y)]$, and define $R_i = R(F_i)$ for $i \in \intr{K}$. 
\end{itemize}

\section{Some preliminary results}
The lemma below shows that for a set $\mathcal{Y} \subseteq \mathbb{R}^d$ and a convex set $\mathcal{X} \subseteq \mathbb{R}^d$, if there exists a function $l:\mathcal{X}\times \mathcal{Y} \to \mathbb{R}$ that is Lipschitz and strongly convex on its first argument, then the function $l$ and the set $\mathcal{X}$ are bounded.

\begin{lemma}\label{assump_cons}
	Let $\mathcal{X} \subseteq \mathbb{R}^d$ be a non-empty convex set, let $\mathcal{Y} \subseteq \mathbb{R}^d$ and $l: \mathcal{X} \times \mathcal{Y} \to \mathbb{R}$ be a function such that for all $y\in \mathcal{Y}$ $l(.,y)$ is $L$-Lipschitz and $\rho$-strongly convex, then we have: 
	\begin{itemize}
		\item $\sup_{x,x' \in \mathcal{X}} \norm{x-x'}\le \frac{B}{L}= 8\frac{L}{\rho^2}$.
		\item $\sup_{x,x' \in \mathcal{X}, y\in \mathcal{Y}} \left| l(x,y) - l(x',y)\right| \le B:=8\frac{L^2}{\rho^2}$
	\end{itemize} 
\end{lemma}

\begin{proof}
	Let $y\in \mathcal{Y}$ and $x_0 , x \in \mathcal{X}$, using the $\rho$-strong convexity of $l(.,y)$ we have:
	\begin{equation*}
	l\left(\frac{x + x_0}{2}, y\right)-\frac{\rho^2}{2} \norm{\frac{x + x_0}{2}}^2 \le \frac{1}{2} \left( l(x_0, y)-\frac{\rho^2}{2} \norm{x_0}^2 \right) + \frac{1}{2} \left( l\left(x, y \right)-\frac{\rho^2}{2} \norm{x}^2 \right) 
	\end{equation*} 
	
	Which implies:
	\begin{equation*}
	\frac{\rho^2}{2} \left( \frac{1}{4} \norm{x_0 + x}^2 - \frac{1}{2} \norm{x_0}^2 - \frac{1}{2} \norm{x}^2\right) \le l\left(\frac{x+x_0}{2}, y\right) - \frac{l(x, y)+l(x_0, y)}{2}.
	\end{equation*}
	
	Using the parallelogram law and the assumption that $l$ is $L$-Lipschitz we have:
	\begin{equation*}
	\frac{\rho^2}{8} \norm{x-x_0}^2 \le L \norm{x-x_0},
	\end{equation*}
	which proves that $\text{diam}\left(\mathcal{X}\right) \le 8 \frac{L}{\rho^2}$. Now using the assumption that $l(.,y)$ is $L$-Lipschitz, we have:
	\begin{align*}
	\left|l(x,y) - l(x_0, y)\right| &\le L\norm{x-x_0} \\
	&\le 8 \frac{L^2}{\rho^2},
	\end{align*}
	which proves the second claim.
\end{proof}

For any $y\in \cY$, let $l^*(y) = \min_{x \in \cX} l(x,y)$, which exists since $l$ is continuous
in $x$ and $\cX$ is a closed
bounded set by the previous lemma, and let $\wt{l}(x,y) :=l(x,y) -l^*(y)$.
By the previous lemma, $\wt{l}(x,y) \in [0,B]$; also, note that the proposed algorithms
remain unchanged if we replace the loss $l$ by $\wt{l}$, since the
algorithms only depend on loss differences
for different predictions $x,x'$ and the same $y$. Similarly, the excess loss
of any predictor remains unchanged when replacing $l$ by $\wt{l}$. Therefore, without loss
of generality we can assume that the loss function always takes values in $[0,B]$, which we
do for the remainder of the paper.

The following lemma is technical, it will be used in the proof of the instance dependent bound (Theorem~\ref{th:main2}).
\begin{lemma}\label{cl:calpure}
	Let $x \ge 1, c \in (0,1)$ and $y >0$ such that:
	\begin{equation}
	\label{eq:hypineq}
	\frac{\log(x/c)}{x} > y.
	\end{equation}
	Then:
	\begin{equation*}
	x < \frac{ 2\log\left( \frac{1}{cy} \right) }{ y}.
	\end{equation*}
\end{lemma}

\begin{proof}
	%Since $x/c >x \geq 3$, it holds $\log(x/c) \geq 1$.
	Inequality~\eqref{eq:hypineq}
	implies
	\[
	x < \frac{\log(x/c)}{y},
	\]
	and further
	\[
	\log(x/c) < \log(1/yc) + \log \log(x/c) \leq \log(1/yc) + \frac{1}{2}\log(x/c),            
	\]
	since it can be easily checked that $\log(t) \leq t/2$ for all $t>0$.
	Solving and plugging back into the previous display leads to the claim.
\end{proof}

\section{Some concentration results}
In this section, we present concentration inequalities for the key quantities used in our analysis. Recall that
Lemma~\ref{assump_cons} shows that under assumption~\ref{assump}, without loss of generality we can assume that
the  loss function takes values in $[0,B]$, $B:=8L^2/\rho^2$.

The following lemma gives the main concentration inequalities we need:
\begin{lemma} \label{lem:lemconc} Suppose Assumption~\ref{assump} holds.
	For any integer $t \ge 1$, and $\delta\in[0,1]$, with probability at least $1 - 3\delta$, for all $i, j \in \intr{K}$:
	\begin{align*}
	\left| \left(\hat{R}_{ij}(i,t) - \hat{R}_{ij}(j,t)\right) - \left(R_i - R_j\right) \right| &\le \sqrt{2}L~\hat{d}_{ij}~\alpha_{ij}(t, \delta) + 3B~\alpha^2_{ij}(t, \delta) \\
	\left| \hat{d}_{ij}^2 - d_{ij}^2\right| &\le \max\left \lbrace 2\frac{B}{L}~\alpha_{ij}(t, \delta)~d_{ij}~;~ 6\left( \frac{B}{L} \right)^2~\alpha^2_{ij}(t, \delta) \right \rbrace \\
	\left|\hat{R}_i(t) - R_i\right| &\le 2B\alpha_{i}(t, \delta).
	\end{align*}
\end{lemma}
\begin{proof}
	The first inequality is a direct consequence of the empirical Bernstein inequality (Theorem 4 in \cite{DBLP:conf/colt/MaurerP09}). Recall that $l$ is $L$-Lipschitz in its first argument. Hence, we have the following bound on the empirical variance of the variable: $l(F_i,Y) - l(F_j,Y)$.
	\begin{align*}
	\widehat{\text{Var}}\left[ l(F_i,Y) - l(F_j,Y)\right] & := \frac{2}{T_{ij}(t) \left( T_{ij}(t) - 1\right)} \sum_{u,v \in \mathcal{T}_{ij}(t)} \left(l\left(F_{i,u},Y_u\right) - l\left(F_{j,u},Y_u\right) - l\left(F_{i,v},Y_v\right) + l\left(F_{j,v},Y_v\right) \right)^2\\
	& \le \frac{1}{T_{ij}(t)}\sum_{u \in \mathcal{T}_{ij}(t)} \left( l\left(F_{i,u},Y_u\right) - l\left(F_{j,u},Y_u\right) \right)^2  \\
	&\le L^2~\hat{d}_{ij}^2.
	\end{align*}
	The second inequality is a consequence of Bernstein inequality applied to $\hat{d}^2_{ij}$, we used the following bound on the variance of the variable $\left( F_i - F_j\right)^2$:
	\begin{align*}
	\text{Var}\left[ \left( F_i - F_j\right)^2\right] & \le \mathbb{E}\left[ \norm{F_i - F_j}^4 \right]\\
	&\le \sup_{i,j \in [K]} \norm{F_i - F_j}^2 \mathbb{E}\left[ \norm{F_i - F_j}^2 \right]\\
	&\le \left(\frac{B}{L}\right)^2 d_{ij}^2.
	\end{align*}
	Finally, the last inequality stems from Hoeffding's inequality. 
\end{proof}
\begin{corollary} \label{cor:deltacontrol}
	Let $T>0$ be fixed. In the full information case $(m=K)$,
	with probability at least $1 - 2\delta$, it holds:
	\begin{align}
	\label{eq:deltacontrol2}
	\text{For all } i, j \in \intr{K}: \qquad 
	\Delta_{ij} & \leq (R_j - R_i) \leq \Delta_{ij} + 32\alpha \max\paren{ L d_{ij}, B \alpha}.          
	\end{align}
\end{corollary}
\begin{proof}
	In the full information case, since all experts are queried at each round we have $T_{ij}(T) = T_i(T)=T$
	and $\alpha_{ij}(T,\delta) = \alpha(T,\delta)=\alpha$ for all $i,j$. Applying
	Lemma~\ref{lem:lemconc} in that setting, using the first inequality we obtain that
	with probability at least $1-3\delta$:
	\[
	\Delta_{ij}  \leq \left(\hat{R}(i,T) - \hat{R}(j,T)\right) -  \sqrt{2}L\hat{d}_{ij}\alpha - 3B\alpha^2
	\leq  R_i - R_j,\]
	giving the first inequality in~\eqref{eq:deltacontrol2}; and
	\begin{align}
	R_i - R_j & \leq \left(\hat{R}(i,T) - \hat{R}(j,T)\right) +  \sqrt{2}L\hat{d}_{ij}\alpha + 3B\alpha^2
	\leq \Delta_{ij} + 9 \alpha L\hat d_{ij} + 9 B\alpha^2. \label{eq:upbdr}
	\end{align}
	From the second inequality in Lemma~\ref{lem:lemconc} we get, putting $\beta:= B/L$:
	\begin{align*}
	\hat{d}_{ij}^2 - d_{ij}^2
	& \le \max\left \lbrace 2\beta \alpha d_{ij}, 6\beta^2\alpha^2 \right \rbrace \\
	& \le \max\left \lbrace 6\beta^2 \alpha^2 + \frac{1}{6}d^2_{ij}, 6\beta^2\alpha^2 \right \rbrace \\
	& \le 6\beta^2 \alpha^2 + \frac{1}{6}d^2_{ij},
	\end{align*}
	from which we deduce $\hat{d}_{ij}^2 \leq 12 \alpha \max(\beta^2 \alpha^2,d^2_{ij})$.
	Taking square roots and plugging into~\eqref{eq:upbdr}, we obtain the claim.
\end{proof}

For $t \ge 1$, define: $\delta_t := \frac{\delta}{t(t+1)}$. 
Define the event $\mathcal{A}$:
\begin{subnumcases}{\label{eq:eventa} \left(\mathcal{A}\right):
	\forall t \ge 1, \forall~i,j \in \intr{K}: }
\left| \left(\hat{R}_{ij}(i,t) - \hat{R}_{ij}(j,t)\right) - \left(R_i - R_j\right) \right| \le 3 \max\left\lbrace L \hat{d}_{ij}~\alpha_{ij}(t, \delta_t) ;B \alpha^2_{ij}(t, \delta_t) \right \rbrace \label{conc1}\\
\left|\hat{R}_i(t) - R_i\right| \le 2B~\alpha_{i}(t, \delta_t) \label{conc2}\\
\hat{d}_{ij}^2 \le 12 \max\left\lbrace d_{ij}^2; \left(\frac{B}{L}\right)^2\alpha^2_{ij}(t, \delta_t) \right \rbrace \label{conc3}\\
d_{ij}^2 \le 12 \max\left\lbrace \hat{d}_{ij}^2; \left(\frac{B}{L}\right)^2\alpha^2_{ij}(t, \delta_t) \right \rbrace \label{conc4}
\end{subnumcases}
Using a union bound over $t\ge 1$ and $i,j \in \intr{K}$, we have: $\mathbb{P}\left(\mathcal{A}\right) \ge 1-4\delta$.

\section{Proof of Theorem~\ref{th:main0} and Corollary~\ref{th:main1}}

Let $t \ge 1$, denote by $S_t$ the set of non-eliminated experts in Algorithm~\ref{algo:bud} at round $t$. The lemma below shows that conditionally to event $\mathcal{A}$, the best experts $\mathcal{S}^*$ are never eliminated.

\begin{lemma}\label{lem:i*}
	If $\mathcal{A}$ defined in~\eqref{eq:eventa} holds, $ \forall t\ge 1$ we have: $\mathcal{S}^{*} \subseteq S_t$, where we recall $\mathcal{S}^{*} := \argmin_{i \in \intr{K}} R(F_i)$.
\end{lemma}

\begin{proof}
	Let $t\ge 1$, assume for the sake of contradiction that:  $i^* \in \mathcal{S}^*$ but $i^* \notin S_t$. Then, at some point, $i^*$ was eliminated by an expert $j$. More specifically: $\exists s \in \intr{t}$, $\exists j \in \intr{K} \setminus \left\lbrace i^* \right \rbrace$, such that $\Delta'_{ji^*}(t, \delta_t) > 0$.
	It follows by definition of $\Delta'_{ji^*}$ that:
	\begin{equation*}
	\hat{R}_{ji^*}(i^*, s) > \hat{R}_{ji^*}(j,s) + 6 \max \left\lbrace L \alpha_{ji^*}( s, \delta_s) \hat{d}_{ji^*}, B \alpha^2_{ji^*}( s, \delta_s) \right \rbrace
	\end{equation*} 
	which contradicts \eqref{conc1} since we have: $R^* \le R_j$.
\end{proof}
The lemma below gives a high probability deviation rate on the excess of any expert in $S_t$ when combined with an appropriate expert. Recall that for $i \in \intr{K}$: $R_i = R(F_i)$. 
\begin{lemma}\label{lem:conc}
	If event $\mathcal{A}$ defined in~\eqref{eq:eventa} holds, $\forall t \ge 1$,  for all $i \in S_t$, let $j \in \text{argmax}_{l \in S_t} \hat{d}_{il}(t)$, then we have:
	\begin{equation*}
	R\left( \frac{F_i + F_j}{2}\right) \le R^* +c~B \frac{\log(K\delta_t^{-1})}{T_{ij}(t)},                
	\end{equation*}
	where $c$ is an absolute constant.
\end{lemma}

\begin{proof}
	Suppose that $\mathcal{A}$ is true. Let $t\ge 1$, $i \in S_t$ and $i^* \in \mathcal{S}^*$. Let $j \in \text{argmax}_{S_t} \hat{d}_{il}$.

	Lemma~\ref{lem:i*} shows that : $i^* \in S_{t}$, we therefore have by construction of Algorithm~\ref{algo:bud}:
	\begin{align*}
	\hat{R}_{ij}(j, t)  &\le \hat{R}_{ij}(i, t) + 6 \max\left \lbrace L \alpha_{ij}(t, \delta_{t}) \hat{d}_{ij}(t), B \alpha^2_{ij}(t, \delta_{t}) \right \rbrace\\
	\hat{R}_{ii^*}(i, t)  &\le \hat{R}_{ii^*}({i^*}, t) + 6 \max\left \lbrace L \alpha_{ii^*}(t, \delta_{t}) \hat{d}_{ii^*}(t), B \alpha^2_{ii^*}(t, \delta_{t}) \right \rbrace.
	\end{align*} 
	Using inequalities \eqref{conc1} for $(i,j)$ and $(i, i^*)$ respectively and $\hat{d}_{ii^*}(t) \le \hat{d}_{ij}(t)$, we have:
	\begin{align}
	R_j  &\le R_i + 9 \max\left \lbrace L \alpha_{ij}(t, \delta_{t}) \hat{d}_{ij}(t), B \alpha^2_{ij}(t, \delta_{t}) \right \rbrace \label{eq:Rj}\\
	R_i  &\le R_{i^*} + 9 \max\left \lbrace L \alpha_{ii^*}(t, \delta_{t}) \hat{d}_{ij}(t), B \alpha^2_{ii^*}(t, \delta_{t}) \right \rbrace \label{eq:Ri}.
	\end{align} 
	We have:
	\begin{align*}
	R\left( \frac{F_i + F_j}{2}\right) &\le \frac{1}{2}\left( R_i - \frac{\rho^2}{2} \mathbb{E}\left[F_i^2\right] \right) + \frac{1}{2} \left(R_j -\frac{\rho^2}{2} \mathbb{E}\left[F_j^2\right]\right)+\frac{\rho^2}{2}\mathbb{E}\left[\left(\frac{F_i+F_j}{2}\right)^2\right]\\
	&= \frac{1}{2} R_i +  \frac{1}{2} R_j - \frac{\rho^2}{8} \left( 2\mathbb{E}\left[F_i^2\right] +2\mathbb{E}\left[F_j^2\right] - \mathbb{E}[ \left(F_i + F_j\right)^2] \right)\\
	&= \frac{1}{2} R_i +  \frac{1}{2} R_j - \frac{\rho^2}{8} d_{ij}^2\\
	&\le \frac{1}{2} R_i + \frac{1}{2} R_i + \frac{9}{2} \max\left \lbrace L \alpha_{ij}(t, \delta_{t}) \hat{d}_{ij}(t), B \alpha^2_{ij}(t, \delta_{t}) \right \rbrace  - \frac{\rho^2}{8} d_{ij}^2 \\
	&= R_i + \frac{9}{2} \max\left \lbrace L\alpha_{ij}(t, \delta_{t}) \hat{d}_{ij}(t), B \alpha^2_{ij}(t, \delta_{t}) \right \rbrace - \frac{\rho^2}{8} d_{ij}^2 \\
	&\le R^* + \frac{27}{2} \max\left \lbrace L \alpha_{ij}(t, \delta_{t}) \hat{d}_{ij}(t), B \alpha^2_{ij}(t, \delta_{t}) \right \rbrace - \frac{\rho^2}{8} d_{ij}^2.
	\end{align*}
	We used the strong convexity of $R$ in the first inequality and we injected \eqref{eq:Rj} to bound $R(F_j)$ in the fourth line and \eqref{eq:Ri} to bound $R(F_i)$ in the  last line.
	Now we use inequality \eqref{conc2} for $(i,j)$ and obtain:
	\begin{align*}
	R\left( \frac{F_i + F_j}{2}\right) - R^* &\le 162 \max\left \lbrace L \alpha_{ij}(t, \delta_{t}) d_{ij}, B \alpha^2_{ij}(t, \delta_{t}) \right \rbrace - \frac{\rho^2}{8} d_{ij}^2 \\
	&\le c~B\alpha^2_{ij}(t, \delta_{t})\\
	&\le c~B\alpha^2_{ij}(t, \delta_t),
	\end{align*}
	where $c$ is an absolute constant. In the final step, we upper bounded the right-hand-side of the first inequality with a parabolic function in $d_{ij}$, then we replaced $d_{ij}$ with the expression achieving the maximum (recall that $B:= 8(L/\rho)^2$). 
	
\end{proof}
%Now we turn to the proof theorem ~\ref{th:main0}. We assume that $\delta$ satisfies: $\log\left(2KT/\delta\right) \le T/4K^2$. Otherwise the bounds in theorem~\ref{th:main0} and corollary~\ref{th:main1} are trivial since:
%\begin{equation*}
%	 \min\left\lbrace \frac{\log(KT\delta^{-1})}{T_{\bar{k}\bar{l}}(T)}; \sqrt{\frac{\log(KT\delta^{-1})}{T_q(T)}} \right\rbrace
%\end{equation*}
\paragraph{Proof of Theorem~\ref{th:main0}.}

Let $T \ge 2K^2$, when Algorithm~\ref{algo:bud} is halted at $T$. Let $\hat{k} \in S_T$ and $\hat{l} \in \text{argmax}_{j \in S_T} \hat{d}_{\hat{k}j}(T)$. 

Let $\hat{q}$ denote the empirical risk minimizer on $S_T$:
\begin{equation*}
\hat{q} \in \argmin_{j \in S_T } \hat{R}_j(T).
\end{equation*} 

We consider two cases. If  $T_{\hat{k}\hat{l}}(T) > \sqrt{T_{\hat{q}}(T)\log\left(K\delta^{-1}_T\right)}$,
then the output of Algorithm~\ref{algo:bud} is $\frac{F_{\hat{k}}+F_{\hat{l}}}{2}$ and we can apply
the bound of Lemma~\ref{lem:conc}.

If  $T_{\hat{k}\hat{l}}(T) \le \sqrt{T_{\hat{q}}(T)\log\left(K\delta^{-1}_T\right)}$,
then the output of Algorithm~\ref{algo:bud} is $F_{\hat{q}}$.  We have:
\begin{align*}
R_{\hat{q}} - R_{i^*} &= R_{\hat{q}} - \hat{R}_{\hat{q}}(T) + \hat{R}_{\hat{q}}(T) - \hat{R}_{i^*}(T) + \hat{R}_{i^*}(T) - R_{i^*}\\
&\le 2B\sqrt{\frac{\log\left(K\delta_T^{-1}\right)}{T_{\hat{q}}(T)}}+ 2B\sqrt{\frac{\log\left(K\delta_T^{-1}\right)}{T_{i^*}(T)}}\\
&\le 2B\sqrt{\frac{\log\left(K\delta_T^{-1}\right)}{T_{\hat{q}}(T)}}+ 2B\sqrt{\frac{\log\left(K\delta_T^{-1}\right)}{T_{\hat{q}}(T)-K}}\\
&\le 5B\sqrt{\frac{\log\left(K\delta_T^{-1}\right)}{T_{\hat{q}}(T)}},
\end{align*}
where we used inequalities \eqref{conc3} for $\hat{q}$ and $i^*$, and the fact that the allocation strategy leads to $\left| T_{i^*}(T) - T_{\hat{q}}(T)\right| \le K$ and $T_i(T) > 2K$ for all $i$. 

%Recall that $\hat{q},i^{*} \in S_T$, suppose that $\delta$ satisfies $\log\left(2\delta\left|S_T\right|/\delta\right) \le \frac{T}{4\left|S_T\right|^2}$. Hence using lemma~\ref{lem:conc_allocation} we have: $T_{\hat{q}}(T) \le 3T_{i^*}(T) $.

As a conclusion we have:
\begin{equation}\label{eq:res_f}
R(\hat{g}) - R_{i^*} \le c~B \min\left\lbrace \frac{\log(KT\delta^{-1})}{T_{\hat{k}\hat{l}}(T)}; \sqrt{\frac{\log(KT\delta^{-1})}{ T_{\hat{q}}(T)}} \right\rbrace,
\end{equation}
where $c$ is an absolute constant.

%If $\delta$ satisfies: $\log\left(2T\left|S_T\right|/\delta\right) \ge \frac{T}{4\left|S_T\right|^2}$, then the bound in \eqref{eq:res_f} becomes trivial since:
%\begin{equation*}
%	1 \le  \min\left\lbrace \frac{\log(KT\delta^{-1})}{T_{\bar{k}\bar{l}}(T)}; \sqrt{\frac{\log(KT\delta^{-1})}{T_q(T)}} \right\rbrace. 
%\end{equation*}

\section{Proof of Theorem~\ref{th:main2}}\label{proof:thm2}
In this section, we prove instance dependent bounds on the number of rounds required to achieve a risk at least as good as the best expert up to $\epsilon > 0$.

The following lemma gives an instance dependent upper and lower bound on the quantities $T_{ij}(t)$, for $i,j \in \intr{K}$.
\begin{lemma}\label{lem:1}
	Let $i,j \in \intr{K}$ such that $R_i \neq R_j$. If $\mathcal{A}$ holds, for all $t \ge 1$,
	if
	\begin{equation*}
	T_{ij}(t) \ge 289 \log\left(K \delta_t^{-1}\right) \max\left\lbrace \frac{L^2d_{ij}^2}{\left|R_i - R_j\right|^2}; \frac{B}{\left|R_i - R_j\right|}\right \rbrace,
	\end{equation*}
	then we have either $\Delta'_{ij} > 0$ or $\Delta'_{ji} > 0$.
	
	Furthermore, if
	\begin{equation*}
	T_{ij}(t) \le  3 \log\left(K \delta_t^{-1}\right) \max\left\lbrace \frac{L^2d_{ij}^2}{\left|R_i - R_j\right|^2}; \frac{B}{\left|R_i - R_j\right|} \right\rbrace,
	\end{equation*}
	then we have $\Delta'_{ij} \le 0$ and $\Delta'_{ji} \le 0$.
\end{lemma}

\begin{proof}
	We start by proving the first claim of the lemma. Let $i,j \in \intr{K}$ and $t \ge 1$ such that:
	\begin{equation}\label{eq:prlem1}
	T_{ij}(t) \ge 289 \log\left(K \delta_t^{-1}\right) \max\left\lbrace \frac{L^2d_{ij}^2}{\left|R_i - R_j\right|^2}; \frac{B}{\left|R_i - R_j\right|}\right \rbrace.
	\end{equation}
	Inequality \eqref{eq:prlem1} implies:
	\begin{equation*}
	\alpha_{ij}\left(t, \delta_t \right) \le \frac{1}{17} \min \left\lbrace \frac{\left|R_i - R_j\right|}{Ld_{ij}}; \sqrt{\frac{\left|R_i - R_j\right|}{B}} \right \rbrace.
	\end{equation*}
	By simple calculus, we see that:
	%\begin{equation*}
	%	\max\left\lbrace 17L~ \alpha_{ij}\left(t, \delta_t \right) d_{ij}~;~ 289~ B\alpha^2_{ij}\left(t, \delta_t \right) \right \rbrace \le \left|R_i - R_j\right| 
	%\end{equation*}
	%Hence:
	\begin{equation*}
	17~ \max\left\lbrace L \alpha_{ij}\left(t, \delta_t \right) d_{ij}~;~ B\alpha^2_{ij}\left(t, \delta_t \right) \right \rbrace \le \left|R_i - R_j\right|. 
	\end{equation*}
	
	Now we use inequality \eqref{conc1} from event $\mathcal{A}$ to upper bound $\left|R_i - R_j\right|$:
	\begin{equation}\label{eq:prlem2}
	17~ \max\left\lbrace L \alpha_{ij}\left(t, \delta_t \right) d_{ij};B \alpha^2_{ij}\left(t, \delta_t \right) \right \rbrace \le \left|\hat{R}_{ij}(i,t) - \hat{R}_{ij}(j,t)\right| + 3 \max\left \lbrace L \alpha_{ij}\left(t, \delta_t \right)\hat{d}_{ij}(t); B\alpha^2_{ij}\left(t, \delta_t \right) \right \rbrace.
	\end{equation}
	Using inequality \eqref{conc2}, we have: 
	\begin{equation*}
	\max \left \lbrace \hat{d}_{ij}(t); \frac{B}{L}\alpha_{ij}\left(t, \delta_t \right) \right \rbrace \le 2\sqrt{3}\max \left \lbrace d_{ij};\frac{B}{L} \alpha_{ij}\left(t, \delta_t \right) \right \rbrace. 
	\end{equation*}
	
	We plug in the inequality above in \eqref{eq:prlem2} and obtain:

	\begin{equation*}
	6 \max\left\lbrace L \alpha_{ij}\left(t, \delta_t \right) \hat{d}_{ij}(t); B\alpha^2_{ij}\left(t, \delta_t \right) \right \rbrace < \left| \hat{R}_{ij}(i,t) - \hat{R}_{ij}(j,t)\right|,
	\end{equation*}	
	implying that we have either $\Delta'_{ij}(t) > 0$ or $\Delta'_{ji}(t) > 0$.
	
	For the second claim, Let $i,j \in \intr{K}$ and $t \in \intr{T}$ such that:
	\begin{equation}\label{upper_t_ij}
	T_{ij}(t) \le 3 \log\left(K\delta_t^{-1}\right) \max\left\lbrace \frac{L^2d_{ij}^2}{\left|R_i - R_j\right|^2}; \frac{B}{\left|R_i - R_j\right|} \right\rbrace.
	\end{equation}
	If $T_{ij}(t) = 0$, then $\Delta'_{ij} = \Delta'_{ji} = - \infty$. 
	
	Otherwise, inequality \eqref{upper_t_ij} implies that:
	\begin{equation*}
	\left|R_i - R_j\right| \le 3 \max\left \lbrace L \alpha_{ij}\left(t, \delta_t \right) d_{ij}~;~B \alpha^2_{ij}\left(t, \delta_t \right) \right \rbrace. 
	\end{equation*} 
	Now we use inequality \eqref{conc1} from event $\mathcal{A}$ to lower bound $\left|R_i - R_j\right|$. We have:
	\begin{equation*}
	\left|\hat{R}_{ij}(i,t) - \hat{R}_{ij}(j,t)\right| - 3 \max\left \lbrace L \alpha_{ij}\left(t, \delta_t \right) \hat{d}_{ij}(t)~;~B \alpha^2_{ij}\left(t, \delta_t \right) \right \rbrace \le 3 \max\left \lbrace L \alpha_{ij}\left(t, \delta_t \right) d_{ij}~;~B \alpha^2_{ij}\left(t, \delta_t \right) \right \rbrace.
	\end{equation*}
	
	We plug in inequality \eqref{conc4} to upper bound $d_{ij}$. We conclude that:
	\begin{equation*}
	\left| \hat{R}_{ij}(i,t) - \hat{R}_{ij}(j,t)\right| \le 6 \max\left\lbrace L \alpha_{ij}\left(t, \delta_t \right) \hat{d}_{ij}(t);B \alpha^2_{ij}\left(t, \delta_t \right) \right \rbrace,
	\end{equation*}
	implying that we have: $\Delta'_{ij}(t) \le 0$ and $\Delta'_{ji}(t) \le 0$.
\end{proof}
Now we turn to the proof of Theorem~\ref{th:main2}. Recall the following notations: for $i \in \intr{K}$ define:
\begin{equation*}
\Lambda_i := \min_{i^* \in \mathcal{S}^*}\max\left\lbrace \frac{L^2d^2_{ii^*}}{\left|R_i - R_{i^*}\right|^2}; \frac{B}{R_i - R_{i^*}} \right \rbrace.
\end{equation*}
Denote the corresponding reordered values: 
\begin{equation*}
\Lambda_{(1)}  \le \Lambda_{(2)} \le \dots \le \Lambda_{(K)} = +\infty,
\end{equation*}  
and $\Lambda^* := \min \{ \Lambda_i; \Lambda_i <+\infty\}$.
%Let $N_\epsilon$ be defined as the smallest integer satisfying:  $\Lambda_{(N_\epsilon)} \le \frac{1}{\epsilon} \le \Lambda_{(N_\epsilon+1)}$ and $\Lambda^* := \max\left\{ \Lambda_{i}; \Lambda_{i} < +\infty\right\}$ (we do not consider the trivial case where all the expert have the same risk). 

\paragraph{Proof of Theorem~\ref{th:main2}.}
By Lemma~\ref{lem:conc}, in order to show that $R(\hat{g}) \le R^* +cB\epsilon$, it suffices to prove that for any $i,j \in S_T$, it holds $T_{ij}(T) \ge B\log(K\delta_T^{-1})/\epsilon$.  

Let $\epsilon > 0$, define the following sequences, for $N \in \intr{K-1}$:
$$
\left\{
\begin{array}{ll}
\phi_N &:= 289(K-N)^2\left( \Lambda_{(N)} - \Lambda_{(N-1)}\right) \log\left(\delta^{-1}C_{\epsilon}\right);\\
\tau_N &:= \sum_{k=1}^{N} \phi_k,
\end{array}
\right.
$$
where we define $\Lambda_{(0)} = 0$ and
\begin{equation*}
	C_{\epsilon} := K\sum_{i \in \mathcal{S}_{\epsilon}^c}\Lambda_{i} + 2 \left|\mathcal{S}_{\epsilon}\right|^2 \min\left\lbrace \frac{1}{\epsilon}, \Lambda^* \right \rbrace.
\end{equation*} 

\begin{claim}\label{cl:elim}
	If event $\mathcal{A}$ holds, for any $N \in \intr{K}$ after round $\ceil{\tau_N}$, all experts $i$ satisfying $\Lambda_{i} \le \Lambda_{(N)}$ are necessarily eliminated. 
\end{claim}
\begin{proof}
	Recall that the number of queries required to eliminate an expert $i\in \intr{K}$ is upper bounded by the number of data points needed to have: $\Delta_{i^*i}>0$ for any $i^* \in \mathcal{S}^*$, which would lead to the elimination of $i$ by $i^*$.

	Let $i^*$ be an arbitrary element of $\mathcal{S}^*$. We use an induction argument, for $N=1$ the claim is a direct consequence of the definition of $\tau_1$ and Lemma~\ref{lem:1}. Let $N < K$ and suppose that the claim is valid for all $i\le N$. Let $j$ denote an expert such that $\Lambda_j = \Lambda_{(N+1)}$ and $j$ was not eliminated before $\ceil{\tau_{N}}$. For $i\le N$, the induction hypothesis suggests that between round $\ceil{\tau_{i}}$ and $\ceil{\tau_{i+1}}$ there was at most $K-i$ non-eliminated experts.
	Since the allocation strategy is uniform over the pairs of experts in $S \times S$, we have:
	\begin{equation}\label{eq:boundtij}
	T_{ji^*}(\tau_{N+1}) \ge 2\sum_{i=0}^{N}\frac{\tau_{i+1} - \tau_{i}}{(K-i)(K-i+1)},
	\end{equation}
	where $\tau_0 = 0$.
	%We use \eqref{eq:boundlambda} to lower bound $\bar{T}_N$, and 
	Recall that the definition of $\tau_{i}$ implies that: 
	\begin{equation}\label{eq:boundlambda}
	\tau_{i+1} - \tau_i = 289 (K-i-1)^2 \log\left(C_{\epsilon} \delta^{-1}\right) \left(\Lambda_{(i+1)} - \Lambda_{(i)}\right).
	\end{equation}
	We plug in the lower bound given in \eqref{eq:boundlambda} into \eqref{eq:boundtij} to obtain:
	\begin{equation*}
	T_{ji^*}(\tau_{N+1}) \ge 289 \log\left(C_{\epsilon} \delta^{-1}\right) \Lambda_{(N+1)}.
	\end{equation*}
	Using Lemma~\ref{lem:1} we conclude that expert $j$ is eliminated before round $\tau_{N+1}$, which completes the induction argument.
	
\end{proof}

\begin{claim}\label{cl:bound_tn}
	We have for any $N \in \intr{K}$:
	
	\begin{equation*}
	\tau_{N} = 289 \log\left(C_{\epsilon}\delta^{-1} \right) \left( \sum_{i=1}^{N-1} (2(K-i) +1) \Lambda_{(i)} + (K-N)^2 \Lambda_{(N)} \right). 
	\end{equation*}
\end{claim}
\begin{proof}
	We have by definition of $\tau_{N}$:
	\begin{align*}
	\tau_{N} &= \sum_{i=1}^{N} \phi_i\\
	&= \sum_{i=1}^{N} 289(K-i)^2\left( \Lambda_{(i)} - \Lambda_{(i-1)}\right) \log\left(\delta^{-1}C_{\epsilon}\right)\\
	&= \sum_{i=1}^{N} 289(K-i)^2 \Lambda_{(i)}  \log\left(\delta^{-1}C_{\epsilon}\right) - \sum_{i=1}^{N} 289(K-i)^2 \Lambda_{(i-1)}  \log\left(\delta^{-1}C_{\epsilon}\right) \\
	&= 289 \log\left(\delta^{-1}C_{\epsilon}\right) \left( \sum_{i=1}^{N-1} (2(K-i) +1) \Lambda_{(i)} + (K-N)^2 \Lambda_{(N)} \right).
	\end{align*}
\end{proof}

\paragraph{Conclusion:}

Let $N_{\epsilon}$ denote the integer satisfying (we do not consider the trivial case where all the expert have the same risk):

\begin{equation*}
\Lambda_{(N_{\epsilon})} < \frac{1}{\epsilon} < \Lambda_{(N_{\epsilon}+1)}. 
\end{equation*}

Recall that we suppose that $T$ satisfies:
\begin{equation*}
T \ge 578 C_{\epsilon} \log(C_{\epsilon}\delta^{-1}).
\end{equation*}
%Define $T_{\epsilon}$ by:
%\begin{equation*}
%T_{\epsilon} = 289 \log(C_{\epsilon} \delta^{-1}) \left( \sum_{i=1}^{N_{\epsilon}-1} (K-i) \Lambda_{(i)} + (K-N_{\epsilon} +1)^2 \Lambda_{(N_{\epsilon})}  \right).
%\end{equation*} 

Observe that (using Claim~\ref{cl:bound_tn}):
\begin{align}
T &\ge \tau_{N_{\epsilon}}  + 289 \log(C_{\epsilon}\delta^{-1}) \left( 2\left|\mathcal{S}_{\epsilon}\right|^2 \min\left\lbrace \frac{1}{\epsilon}; \Lambda^* \right \rbrace - (K-N_{\epsilon} )^2 \Lambda_{(N_{\epsilon})} \right)\\
&\ge \tau_{N_{\epsilon}}  + 289 \log(C_{\epsilon}\delta^{-1}) \left( 2\left|\mathcal{S}_{\epsilon}\right|^2 \min\left\lbrace \frac{1}{\epsilon}; \Lambda^* \right \rbrace - \left| \mathcal{S}_{\epsilon}\right|^2 \Lambda^* \right)\\
&\ge \tau_{N_{\epsilon}} + 289 \log(C_{\epsilon}\delta^{-1}) \left|\mathcal{S}_{\epsilon}\right|^2 \min\left\lbrace \frac{1}{\epsilon}; \Lambda^* \right \rbrace. \label{eq:obs}
\end{align}

Claims~\ref{cl:elim} and~\ref{cl:bound_tn} show that after $\ceil{\tau_{N_{\epsilon}}}$ rounds only elements $i \in \intr{K}$ satisfying: $\Lambda_{i} \le \Lambda_{(N_{\epsilon})}$ are eliminated. Therefore, if $1/\epsilon > \Lambda^*$, we have : $\Lambda_{(N_{\epsilon})} = \Lambda^*$ and all the remaining experts are optimal (i.e. in $\mathcal{S}^*$). Hence the mean of any two experts in $\mathcal{S}$ satisfies: $R(\hat{g}) \le R^{*}$.

Now suppose that $1/\epsilon < \Lambda^*$. We have for the last $T-\ceil{\tau_{N_{\epsilon}}}$ rounds all the experts in $\mathcal{S}_{\epsilon}^c$ were eliminated (hence there was at most $\left| \mathcal{S}_{\epsilon} \right|$ non-eliminated experts). Let $(\hat{k}, \hat{l})$ denote the pair output by algorithm~\ref{algo:bud} after $T$ rounds, we have:

\begin{align*}
T_{\hat{k}\hat{l}}(T) &\ge \log(C_{\epsilon}\delta^{-1}) \frac{T - \tau_{N_{\epsilon}}}{\left|\mathcal{S}_{\epsilon}\right|^2} \\
&\ge 289 \frac{\log(C_{\epsilon}\delta^{-1})}{\epsilon} \\
&\ge c\log(KT\delta^{-1}) \frac{1}{\epsilon},
\end{align*}
where $c$ is a numerical constant, we used \eqref{eq:obs} for the second line, and a simple calculation to obtain the last line. Using Lemma~\ref{lem:conc}, we obtain the desired conclusion.

\section{Proof of Theorem~\ref{th:main3}}

In this section we will show that for $C$ large enough, if $\mathcal{A}$ holds, we have:
\begin{equation}\label{eq:goal}
R(\hat{g}) - R^* \lesssim \epsilon.
\end{equation}

Let $i^*$ be an arbitrary element of $\mathcal{S}^*$. Denote $T_i$ the number of queries required to eliminate an expert $i \in \intr{K}$. $T_i$ is upper bounded by the number of data points needed to have: $\Delta_{i^*i} >0 $, which would lead to the elimination of $i$ by $i^*$. The following claim, which is a consequence of Lemma~\ref{lem:1}, provides this upper bound.

\begin{claim}\label{cl:1bis}
	If $\mathcal{A}$ holds, let $i \in \intr{K}$ be a suboptimal expert ($\Lambda_i < +\infty$). We have:
	\begin{equation*}
	T_i \le 289\log\left(KC\delta^{-1}\right) \Lambda_{i}.
	\end{equation*}	
\end{claim}
\begin{proof}
	Lemma~\ref{lem:i*} shows that experts $i^* \in \mathcal{S}^*$ are never eliminated if $\mathcal{A}$ is true. Using Lemma~\ref{lem:1}, the number of queries required for the elimination of a suboptimal expert $i$ by expert $i^*$, satisfies:
	\begin{equation*}
	T_i \le 289\log\left(KC\delta^{-1}\right) \Lambda_{i}.
	\end{equation*}
\end{proof}

Let $\epsilon \ge 0$. Recall that $\mathcal{S}_{\epsilon}$ is defined by:
\begin{equation*}
	\mathcal{S}_{\epsilon} := \left\lbrace  i \in \intr{K}: \Lambda_i > \frac{1}{\epsilon}\right\rbrace
\end{equation*}
Suppose that we have:
\begin{equation*}
	C > 578 \left( \sum_{i \in \mathcal{S}_{\epsilon}^c}\Lambda_{i} + \left|\mathcal{S}_{\epsilon}\right| \min\left\lbrace \frac{1}{\epsilon}; \Lambda^* \right \rbrace \right)   \log\left(K\delta^{-1} \left( \sum_{i \in \mathcal{S}_{\epsilon}^c}\Lambda_{i} + \left|\mathcal{S}_{\epsilon}\right| \min\left\lbrace \frac{1}{\epsilon}; \Lambda^* \right \rbrace  \right)\right),
\end{equation*}
%where $\kappa$ is the constant in lemma~\ref{lem:conc} depending only on $\eta, L$ and $\rho$.
We therefore have using Lemma~\ref{cl:calpure}:

\begin{equation*}
	C > 289\log\left(KC\delta^{-1}\right) \left( \sum_{i \in \mathcal{S}_{\epsilon}^c}\Lambda_{i} + \left|\mathcal{S}_{\epsilon}\right| \min\left\lbrace \frac{1}{\epsilon}; \Lambda^* \right \rbrace \right).
\end{equation*} 

Let us denote by $C_1$ the total number of queries received by all the experts in $\mathcal{S}_{\epsilon}$ and by $C_2$ the total number of queries received by the remaining experts. We therefore have: $C = C_1+C_2$. In order to show that at a certain round, all the experts in $\mathcal{S}_{\epsilon}^c$ were eliminated, it suffices to prove that:
\begin{equation*}
	C_1 \ge \left|\mathcal{S}_{\epsilon} \right| \max_{i \in \mathcal{S}_{\epsilon}^c} T_i,
\end{equation*}
since the inequality above shows that the budget is not totally consumed after round $\max_{i \in \mathcal{S}_{\epsilon}^c} T_i $ where all elements in $\mathcal{S}_{\epsilon}^c$ where eliminated.

Claim~\ref{cl:1bis} provides the following upper bound for $C_2$:
\begin{equation*}
	C_2 \le 289\log\left(KC\delta^{-1}\right)~\sum_{i \in \mathcal{S}_{\epsilon}^c} \Lambda_i.
\end{equation*} 
We therefore have:
\begin{align*}
	C_1 &= C - C_2 \\
	&\ge 289\log\left(KC\delta^{-1}\right)~\left( \sum_{i \in \mathcal{S}_{\epsilon}^c} \Lambda_i + \left|\mathcal{S}_{\epsilon}\right| \min\left\lbrace \frac{1}{\epsilon}; \Lambda^* \right \rbrace\right)-C_2\\
	&\ge 289\log\left(KC\delta^{-1}\right)~\left( \sum_{i \in \mathcal{S}_{\epsilon}^c} \Lambda_i + \left|\mathcal{S}_{\epsilon}\right| \min\left\lbrace \frac{1}{\epsilon}; \Lambda^* \right \rbrace\right)-289\log\left(KC\delta^{-1}\right)~\sum_{i \in \mathcal{S}_{\epsilon}^c} \Lambda_i.
\end{align*}
Hence:
\begin{equation}\label{eq:c1}
	C_1 \ge 289\log\left(KC\delta^{-1}\right) \left|\mathcal{S}_{\epsilon}\right| \min\left\lbrace \frac{1}{\epsilon}; \Lambda^* \right \rbrace
\end{equation}
Recall that by definition of $\mathcal{S}_{\epsilon}$, using Claim~\ref{cl:1bis} we have: 
\begin{equation*}
	\max_{i \in \mathcal{S}_{\epsilon}^c} T_i \le 289\log\left(KC\delta^{-1}\right)\min\left\lbrace \frac{1}{\epsilon}; \Lambda^* \right \rbrace,
\end{equation*}
hence:
\begin{equation*}
	C_1 \ge \left|\mathcal{S}_{\epsilon} \right| \max_{i \in \mathcal{S}_{\epsilon}^c} T_i.
\end{equation*}

This shows that $S \subseteq \mathcal{S}_{\epsilon}$. We have two possibilities: if $\frac{1}{\epsilon} < \Lambda^*$, the selected pair $(F_{\bar{k}}, F_{\bar{l}}) \in S \times S$ satisfies:
\begin{equation*}
	T_{\bar{k}\bar{l}} = \min\left \lbrace T_{\bar{k}}, T_{\bar{l}} \right \rbrace \ge \frac{C_1}{\left| \mathcal{S}_{\epsilon}\right|}.
\end{equation*}
Using \eqref{eq:c1}, we have:
\begin{equation}\label{eq:c2}
	T_{\bar{k}\bar{l}} \ge 289 \log\left(KC\delta^{-1}\right) \frac{1}{\epsilon}. 
\end{equation}
Observe that Lemma~\ref{lem:conc} applies in this setting. In particular, the total number of rounds $T$ of algorithm~\ref{algo:budgeted}, satisfy: $T \le C$. Hence, it holds
\begin{equation*}
	R\left(\frac{F_{\hat{k}}+F_{\hat{l}}}{2}\right) - R^* \le c~B \frac{\log(KC\delta^{-1})}{T_{\bar{k}\bar{l}}}.
\end{equation*}
We conclude by injecting inequality \eqref{eq:c2} in the bound above. We therefore have: 
\begin{equation*}
	R\left(\hat{g}\right) - R^* \le cB~\epsilon,
\end{equation*}
where $c$ is an absolute constant.

If $\frac{1}{\epsilon} > \Lambda^*$, by definition of $\Lambda^*$ and the fact that $\mathcal{S} \subseteq \mathcal{S}_{\epsilon}$, we conclude that only the optimal experts (i.e. the experts $i$ such that $R_i = R^*$) remain when the budget is totally consumed. Hence combining any 2 of these expert will lead to the bound: $R\left(\hat{g}\right) \le R^* $.

\section{Proof of lower bounds}

The lemma below  gives a lower bound for the problem of estimating the parameter describing a Bernoulli random variable.

\begin{lemma}[\cite{anthony2009neural}, Lemma~5.1] \label{lower:b} 
	Suppose that $\alpha$ is a random variable uniformly distributed on $\left\{ \alpha_{-}, \alpha_{+}\right\}$, where $\alpha_{-} = 1/2-\epsilon/2$ and $\alpha_{+} = 1/2+\epsilon/2$, with $0<\epsilon<1$. Suppose that $\xi_1,\dots, \xi_{m}$ are i.i.d $\left\{0,1\right\}$-valued random variables with $\mathbb{P}\left( \xi_i = 1\right) = \alpha$ for all $i$. Let $f$ be a function from $\left\{0,1\right\} \to \left\{ \alpha_{-}, \alpha_{+}\right\}$. Then it holds:
	\begin{equation*}
	\mathbb{P}\left(f\left(\xi_1, \dots, \xi_{m}\right) \neq \alpha \right) > \frac{1}{4} \left(1 - \sqrt{1-\exp\left( \frac{-2 \ceil{m/2}\epsilon^2}{1-\epsilon^2}\right)}\right).
	\end{equation*}
\end{lemma}

\subsection{Proof of Lemma~\ref{lem:lower1}}
Let $T>0$ and consider an convex combination of experts $\hat{g}$ output after full observation of $T$ training rounds. We will construct two experts $F_1$ and $F_2$ and a target variable $Y$ and we will show that, for these variables, a strategy for our problem ($m=2$ and $p=1$) gives a solution to the problem in Lemma~\ref{lower:b}.
Finally we will use the lower bound from this lemma.

For $\theta\in[0,1]$, let $\mbp_\theta$ denote the probability distribution
of $T$ i.i.d. draws $Y_1,\ldots,Y_T$ of Bernoulli variables or parameter $\theta$,
% be i.i.d $\{0,1\}$-valued random variables such that $\mathbb{P}\left( Y_i = 1\right) = \alpha$ for all $i$. Let
while $F_{1,t} = 0$ and $F_{2,t} = 1$ almost surely for $t\in \intr{T}$. 
Let $\alpha$ be a variable that is uniformly distributed on $\{ \alpha_-, \alpha_+\}$
with $\alpha_\pm= \frac{1}{2} \pm \frac{\epsilon}{2}$, and $\epsilon \in (0,1)$ is a parameter to be tuned subsequently; let the training obervations be drawn according to $\mbp_\alpha$.
Since $p=1$, the output $\hat{g}$ is either $F_1$ or $F_2$. Define $f:\{0,1\}^T \to \{ \alpha_- , \alpha_+\}$ such that given $(Y_1, \dots, Y_T)$, $f$ outputs $\frac{1}{2} - \frac{\epsilon}{2}$ if $\hat{g} = F_1$ and $\frac{1}{2} + \frac{\epsilon}{2}$ if $\hat{g} = F_2$. By construction we have that the events $\{f = \alpha\}$ and $\{R(\hat{g})= \min\{R_1, R_2\}\}$ are equivalent. Using Lemma~\ref{lower:b} and setting $\epsilon = \frac{c_0}{\sqrt{T}}$ where $c_0$ is a constant such that the lower bound in Lemma~\ref{lower:b} is equal to 0.1, we have:
\begin{equation*}
\mathbb{P}\left( R(\hat{g}) - \min\left\{R_1, R_2\right\} \ge \frac{c_0}{\sqrt{T}}\right) > 0.1.
\end{equation*}
Due to the randomization of $\alpha$, the  above probability  is the average of the corresponding
event under $\mbp_{\alpha_-}$ and $\mbp_{\alpha_+}$. Therefore, under at least one of these two
training distributions, the deviation event has a probability at least $0.05$.

\subsection{Proof of Lemma~\ref{lem:lower2}}
The gist of the proof is the following. We will construct a distribution with two experts that
are very correlated. In this situation, going from a weighted average of the two experts to
a single expert with the largest weight does not change the prediction risk much, and so we could
find a single expert with small risk if the weighted average has small risk. On the other hand,
since the agent only observes one expert per training round, from their point of view the observational
distribution is identical as if the experts were independent -- the correlation cannot be
observed. Therefore the same strategy could be used to find the best expert in the independent case.
This contradicts the lower bounds in this case (which is a standard bandit setting), therefore
it is impossible to pick consistently a weighted average with small risk in a situation where
the correlations cannot be observed.

Let $T>0$ be fixed. We consider the particular setting
where the target variable $Y$ is
identically $0$, and the expert predictions $F_1$ and $F_2$ are
two (non independent) Bernoulli random variables.
We define a distribution $\mbp_-$ for $(F_1,F_2)$ such that:
\begin{itemize}
	\item the marginal distribution of $F_1$ is Bernoulli of parameter $\alpha_-=\frac{1}{2} - \frac{\epsilon}{2}$;
	\item the marginal distribution of $F_2$ is Bernoulli of parameter $\alpha_+=\frac{1}{2} + \frac{\epsilon}{2}$;
	\item it holds that $\mbp_-(F_1F_2=1) = \alpha_-$.          
\end{itemize}
Note that this can be easily constructed as $F_1= \ind{U \leq \alpha_-}; F_2= \ind{U \leq \alpha_+}$,
where $U$ is a uniform variable on $[0,1]$. Let $\mbp_+$ be defined similarly with the role
of $F_1$ and $F_2$ reversed. Here, $\epsilon$ is a positive parameter to be tuned later.
We denote $R_-,R_+$ for the prediction risks under distributions $\mbp_-,\mbp_+$.
We have $R_-(F_1) = R_+(F_2) = \alpha_-$, $R_-(F_2) = R_+(F_1) = \alpha_+$, and
$R^* = \alpha_-$ is the same under $\mbp_-$ and $\mbp_+$.
%	Let $\mathbb{P} = \left(\mathcal{B}\left(\frac{1}{2}-\frac{\epsilon}{2}\right); \mathcal{B}\left(\frac{1}{2}+\frac{\epsilon}{2}\right)\right)$ and $\mathbb{P}^{'} = \left(\mathcal{B}\left(\frac{1}{2}+\frac{\epsilon}{2}\right); \mathcal{B}\left(\frac{1}{2}-\frac{\epsilon}{2}\right)\right)$be the probability distributions for the experts advices $(F_1,F_2)$ such that: $\mathbb{P}\left(F_1F_2=1\right) = \mathbb{P}^{'}\left(F_1F_2=1\right) = \frac{1}{2} - \frac{\epsilon}{2}$, where $\epsilon$ 

Let us be given an arbitrary training observation strategy $\pi$ (prescribing at each training round which expert to
observe based only on past observations), and output a convex combination of experts $\hat{g}$. This output is a convex combination of $F_1$ and $F_2$, hence it is characterized by the weight of $F_1$, which we denote $\hat{\alpha}$. The parameter $\hat{\alpha}$ depends on the observed data. %We associate to the strategy $\left(\pi, \phi, \hat{g}\right)$ the function
We %$f$ that takes as input the data observed by the forecaster using $\pi$
also define $\hat{f}$ associated to this training strategy, that outputs $F_1$ if $\hat{\alpha}> \frac{1}{2}$ and $F_2$ otherwise.
Finally, let us denote $\mbq_{\pi}^+$ the distribution of the training data observed by
the agent when the $T$ experts opinions are drawn i.i.d. from $\mbp_-$ and the agent observes
the expert advices following strategy $\pi$; and define $\mbq_{\pi}^{-}$ similarly.

Define the event $\mathcal{A}_+ := \left\{ R_+\left(\hat{g}\right) - R^* \ge \frac{1}{4} \epsilon \right\}$ %$\left|R_+(F_1) - R_+(F_2)\right| \right\}$
and similarly $\mathcal{A}_-$.
In the remainder of the proof, we will show, using Bretagnolle-Hubert inequality (Theorem 14.2 in \cite{lattimore2020bandit}), that either $\mbq_{\pi}^{-}(\cA_-)$ or $\mbq_{\pi}^{+}(\cA_+)$ is
% the probability of the event: $\mathcal{A} := \left\{ R\left(\hat{g}\right) - R^* \ge \frac{1}{4}\left|R(F_1) - R(F_2)\right| \right\}$, is
lower bounded by a positive constant. %either under $\mbq_{\pi}^{-}$ or $\mbq_{\pi}^{+}$.

We have under the distribution $\mathbb{P}_-$:
\begin{align*}
R_-\left(\hat{g}\right) - R_-(\hat f) &= \mathbb{E}_-\left[ \left(\hat{\alpha} F_1 + (1-\hat\alpha) F_2\right)^2 \right] - \mathbb{E}_-\left[ \left( \mathds{1}\left(\hat{\alpha} > \frac{1}{2}\right)F_1 + \mathds{1}\left(\hat{\alpha} \le \frac{1}{2} \right)F_2 \right)^2\right] \\
&= \epsilon (1 - \hat{\alpha})^2 - \epsilon \left( 1 - \mathds{1} \left( \hat{\alpha} > \frac{1}{2}\right) \right) \\
&\ge -\frac{3}{4} \epsilon. 
\end{align*}
%	Similarly, under distribution $\mathbb{P}_+$:	
%	\begin{align*}\label{eq:f}
%	R_+(\hat{g}) - R_+(\hat f) &= \epsilon \hat{\alpha}^2  - \epsilon \mathbbm{1} \left( \hat{\alpha} > \frac{1}{2} \right)  
%	\ge -\frac{3}{4} \epsilon.
%	\end{align*}	
%Since $\abs{R_-(F_1)-R_-(F_2)}=\epsilon$,
Note that the above estimate crucially depends on the fact that $F_1,F_2$ are not independent
under $\mbp_-$. In view of the
above, the event $\mathcal{A}_-$ is implied by $R_-(\hat f) - R^* = \epsilon$.
Similarly, $\mathcal{A}_+$ is implied by $R_+(\hat f) - R^* = \epsilon$.
%under $\mbp_-$ as well as $\mbp_+$,
Hence:
\begin{align*}
\mbq_{\pi}^-\left( \mathcal{A}_-  \right) + \mathbb{Q}_\pi^+\left( \mathcal{A}_+ \right) &\ge \mbq_{\pi}^-\left( R_-(\hat f) - R^* = \epsilon\right) + \mbq_{\pi}^+\left(R_+(\hat f) - R^* = \epsilon \right)\\
&= \mbq_{\pi}^- \left( \hat f = F_2\right) + \mbq_{\pi}^+ \left(\hat f \neq F_2 \right).
\end{align*}
Now we use Bretagnolle-Hubert inequality:
\begin{equation*}
\mbq_{\pi}^- \left( f = F_2\right) + \mbq_{\pi}^+\left(f \neq F_2 \right) \ge \frac{1}{2} \exp\left( -D\left(\mbq_{\pi}^-, \mbq_{\pi}^+\right)\right),
\end{equation*}
where $D(\mbq_{\pi}^-, \mbq_{\pi}^+)$ is the relative entropy between $\mbq_{\pi}^-$ and $\mbq_{\pi}^+$. In order to conclude, we need an upper bound on $D(\mbq_{\pi}^-, \mbq_{\pi}^+)$.
Since the agent only observes one expert in each round according to strategy $\pi$,
the distribution of the observed data $\mbq_{\pi}^-$ or $\mbq_{\pi}^+$ is unchanged
if we replace the generating distributions $\mbp_-$ or $\mbp_+$ by distributions having the
same marginals, but for which $F_1$ and $F_2$ are independent. Therefore, 
the observational distributions  $\mbq_{\pi}^-, \mbq_{\pi}^+$ are equivalent to that of
the observational distributions, under the same strategy, of a canonical bandit model with two arms.
%
%	Recall that for each $t \in [T]$, the forecaster has to choose one expert in order to observe its advice according to his strategy $\pi$. Let $\mathbb{Q} = \mathbb{P}_{\pi}$ and $\mathbb{Q'} = \mathbb{P'}_{\pi}$ be the probability measures on a canonical bandit model (\textcolor{red}{specify}) induced by the $T$-round interconnection of $\pi$ and $\mathbb{P}$ ($\pi$ and $\mathbb{P'}$ respectively).
We can then use the divergence decomposition formula (Lemma 15.1 of \cite{lattimore2020bandit}) to upper bound $D\left(\mbq_{\pi}^-, \mbq_{\pi}^+\right)$; denoting $\mbp_-^{(1)}$, $\mbp_-^{(2)}$ the marginals
of $\mbp_-$ and similarly for $\mbp_+$, it holds
% (recall that $m=1$ hence one expert can be seen in each round):
\begin{equation*}
D\left(\mbq_{\pi}^-, \mbq_{\pi}^+\right) = \mathbb{E}_{-}[T_1] D(\mbp_-^{(1)}, \mbp_+^{(1)}) + \mathbb{E}_{-}[T_2] D(\mbp_-^{(2)}, \mbp_+^{(2)}),
\end{equation*}
where the expectation $\mathbb{E}_{-}[.]$ is with respect to the probability distribution $\mbq_{\pi}^-$ and $T_i$  denotes the total number of rounds where the advice of expert $F_i$ was queried using the strategy $\pi$.
We have: $T_1+ T_2= T$ almost surely, and $D(\mbp_-^{(1)}, \mbp_+^{(1)})=D(\mbp_-^{(2)}, \mbp_+^{(2)}) \le 4 \epsilon^2$ provided $\epsilon \leq \frac{1}{2}$. Therefore:
\begin{equation*}
\mbq_{\pi}^-\left( \mathcal{A}_-  \right) + \mbq_{\pi}^+\left( \mathcal{A}_+ \right) \ge \frac{1}{2} \exp\left(-4\epsilon^2 T\right).
\end{equation*}
This shows that there exists a probability distribution $\mathbb{P} \in \{ \mathbb{P}_-, \mathbb{P}_+\}$ for the experts advices and the target variable such that the prediction $\hat g$ satisfies:
\begin{equation*}
\mathbb{P}\left( R(\hat{g}) - R^*  \ge \epsilon \right) \ge \exp\left(-4\epsilon^2 T\right),
\end{equation*} 
We conclude by choosing $\epsilon = \frac{1}{2\sqrt{T}}$.

\section{Intermediate case: $m\ge 3, p=2$}\label{sec:m3}

In this section we assume that the learner is allowed to access more than two experts advices per round. We show that this leads to an improvement of the bound in Theorem~\ref{th:main1}. We consider the following extension of Algorithm~\ref{algo:bud}:

\begin{algorithm}[H] 
	%\centering
	\caption{Intermediate case \label{algo:bud3}}
	\begin{algorithmic}
		\STATE \textbf{Input} $m$, $L$ and $\rho$.
		\STATE Initialization: $S \gets \intr{K}$. 
		\FOR{ $T=1,2,\dots $ } 
		\STATE Sample a subset $\mathcal{M}$ of size $m$ from $\intr{K}$ uniformly at random.
		\STATE Query the advice of experts in $\mathcal{M}$ and update the corresponding quantities.
		\STATE For all $i,j$: If $\Delta'_{ij}>0$: $S \gets S \setminus \{j\}$.
		\ENDFOR  
		\STATE \textbf{On interrupt:} Let $\hat{k} \in S$ and let $\hat{l} \gets \underset{j \in S}{\text{argmax}}~\hat{d}_{\hat{k}j}$.
		\STATE Return $\frac{1}{2} \left(F_{\hat{k}} + F_{\hat{l}}\right)$.
	\end{algorithmic}
\end{algorithm}

\begin{theorem}\label{th:maingen}(Instance independent bound)
	Suppose Assumption~\ref{assump} holds. Let $T \ge 1$, and denote $\hat{g}$ the output of Algorithm~\ref{algo:bud3} with inputs $(m, L, \rho)$ in round $T$. If $m\ge 3$, then with probability at least $1-\delta$:
	\begin{equation*}
	R\left(\hat{g}\right) \le \min_{i \in \intr{K}} R_i + cB~  \frac{(K/m)^2\log\left(2TK\delta^{-1}\right)}{T},
	\end{equation*}
	where $c$ is an absolute constant. 
\end{theorem}

\begin{proof}
	Let $i,j \in \intr{K}$, denote $T_{ij}(T)$ the total number of rounds where the advice of expert $i$ and $j$ were jointly queried. We have: $T_{ij}(T) = \sum_{t=1}^{T} \mathds{1}\{ i \text{ and } j \text{ were jointly queried at round } t \}$. We conclude that $T_{ij}(T)$ is the sum of $T$ independent and identically distributed Bernoulli variables with parameter: $\frac{m(m-1)}{K(K-1)}$. We therefore have the following consequence of Bernstein concentration inequality, with probability at least $1-\delta$, for all $i,j \in \intr{K}$ and $T\ge K$:
	\begin{equation}\label{eq:conc31}
	\left|T_{ij}(T) - \mathbb{E}\left[T_{ij}(T)\right]\right| \le \sqrt{2T\frac{m(m-1)}{K(K-1)} \log(2KT/\delta)} + \frac{1}{3} \log(2KT/\delta). 
	\end{equation}
	%Moreover, denote $T_i(T)$ the total number of round where the advice of expert $i$ was queried. We have: $T_{i}(T) = \sum_{t=1}^{T} \mathbbm{1}\{ \text{i was queried at t} \}$. Therefore: $T_i(T)$ is a sum of Bernoulli independent and identically distributed Bernoulli variables with parameter: $\frac{m}{K}$. The following is a consequence of Bernstein concentration inequality, with probability at least $1-\delta$, for all $i \in \intr{K}$ and $T\ge K$:
	%\begin{equation}\label{eq:conc32}
	%\left|T_i(T) - \mathbb{E}\left[T_i(T)\right]\right| \le \sqrt{2T\frac{m}{K} \log(2KT/\delta)} + \frac{1}{3} \log(2KT/\delta)
	%\end{equation}
	Suppose that $\delta$ satisfies:
	\begin{equation*}
	\log(2KT/\delta) \le \frac{1}{16} \frac{m^2}{K^2}T. 
	\end{equation*}
	Then we have:
	\begin{equation}\label{eq:impl1}
	\sqrt{2T\frac{m(m-1)}{K(K-1)} \log(2KT/\delta)} + \frac{1}{3} \log(2KT/\delta) \le \frac{1}{2}\frac{m(m-1)}{K(K-1)} T,
	\end{equation}
	%and
	%\begin{equation}\label{eq:impl2}
	%\sqrt{2T\frac{m}{K} \log(2KT/\delta)} + \frac{1}{3} \log(2KT/\delta) \le \frac{1}{2}\frac{m}{K} T.
	%\end{equation}

	Observe that the result of Lemma~\ref{lem:conc} still holds in this setting for non-eliminated elements (experts in $S_T$), since the elimination criterion for an expert $j$, which consists of the existence of $i$ such that $\Delta'_{ij}>0$, is the same as in Algorithm~\ref{algo:bud}. Let $\hat{g}$ denote the output of Algorithm~\ref{algo:bud3}, we conclude that if $\mathcal{A}$ and \eqref{eq:conc31} hold for all $i,j$ and $T$, we have:
	\begin{equation}\label{eq:conc_fin0}
	R(\hat{g}) - R_{i^*} \le \kappa  \frac{\log\left(KT\delta^{-1}\right)}{T_{\hat{k}\hat{l}}(T)},
	\end{equation}
	where $\kappa$ is a constant depending only $\eta, L$ and $\rho$. 
	Finally, we use \eqref{eq:impl1}. We therefore have with probability at least $1 - 4\delta$:
	\begin{equation*}
	R\left(\hat{g}\right) \le \min_{i \in \intr{K}} R_i + c~B  \frac{(K/m)^2\log\left(2TK\delta^{-1}\right)}{T}.
	\end{equation*}
	Now suppose that $\delta$ satisfies:
	\begin{equation*}
	\log(2KT/\delta) \ge \frac{1}{16} \frac{m^2}{K^2}T,
	\end{equation*}
	then it holds:
	\begin{equation*}
	\frac{(K/m)^2\log\left(2TK\delta^{-1}\right)}{T} \ge \frac{1}{16}. 
	\end{equation*}
	We conclude that for $\bar{c} = \max\{c, 16 \}$ we have:
	\begin{equation*}
	R\left(\hat{g}\right) - \min_{i \in \intr{K}} R_i \le  B \le \bar{c} B  \frac{(K/m)^2\log\left(2TK\delta^{-1}\right)}{T}.
	\end{equation*}
	
\end{proof}

\end{document}